\documentclass[11pt]{amsart}
\usepackage{amssymb,latexsym}
\newtheorem{theorem}{Theorem}[section]
\newtheorem{claim}[theorem]{Claim}
\newtheorem{prop}[theorem]{Proposition}
\newtheorem{lemma}[theorem]{Lemma}
\newtheorem{remark}[theorem]{Remark}

\newtheorem{definition}[theorem]{Definition}
\newtheorem{cor}[theorem]{Corollary}

\begin{document}

\title{Additivity and Relative Kodaira Dimensions}
\author{Tian-Jun Li }
\address{School  of Mathematics\\  University of Minnesota\\ Minneapolis, MN 55455}
\email{tjli@math.umn.edu}
\author{Weiyi Zhang}
\address{School  of Mathematics\\  University of Minnesota\\ Minneapolis, MN 55455}
\email{zhang393@math.umn.edu}

\maketitle

\begin{center}
\textit{Dedicated to Shing-Tung Yau on the occasion of his $60^{th}$
birthday}\end{center}

\section{Introduction}
The notion of Kodaira dimension has been defined for complex
manifolds in \cite{K}, for symplectic $4-$manifolds in \cite{L1}
(see also \cite{McS}, \cite{LeB}). It is shown in \cite{DZ} (and
\cite{L1}) that these two definitions are compatible in dimension
$4$. Furthermore, we calculate it for some (4-dimensional) Lefschetz
fibrations when the base has positive genus. In \cite{Z}, this
notion is extended to $3-$dimensional manifolds
via geometric structures in the sense of Thurston. All these Kodaira
dimensions are ``absolute'' invariants, taking values in the set
\begin{equation}\label{value} \{-\infty, 0, 1, \cdots, \lfloor\frac{n}{2}\rfloor\},\end{equation} where
$n$ is the real dimension of the manifold, and $\lfloor x\rfloor$ is
the largest integer bounded by $x$.

We will review them and introduce a few more for logical convenience
in section $2$. We also point out  in section 2 that they are
invariant under covering,
and further show that they are additive for many fiber bundles.

In  recent years, the study of relative  invariants for a pair of
symplectic (projective) manifold with a codimension 2 symplectic
submanifold (smooth divisor) becomes increasingly important,
especially in Gromov-Witten theory (\cite{aLR}, \cite{MP},
\cite{EGH}, \cite{jLi}).
 The  relative
invariants are used to calculate the absolute invariants via fiber
sum and its reverse, symplectic cut (degeneration).
 From this point of view, the paper \cite{LY} by the first author and Yau can
be viewed as a first step towards a possible definition of relative
Kodaira dimension for symplectic $4-$manifolds.


In this paper, motivated by \cite{LY}, we  introduce in section 3
the notion of relative Kodaira dimension $\kappa^s(M^4, \omega,
F^2)$ for  a symplectic $4-$manifold $(M^4, \omega)$ with a possibly
disconnected embedded symplectic surface $F$.
 They take the same set of values as in \eqref{value}.

 To define it we need to establish several homological properties of
 embedded symplectic surfaces in 3.1-2. These properties are formulated in
 terms of the formal Kodaira dimension \eqref{kappa(r)}.
It should be mentioned that symplectic spheres do not satisfy most
of these properties.
 We also
 formulate in 3.3 the notion of relative minimal model, and prove
 the existence and the somewhat surprising uniqueness.

We then define $\kappa^s(M^4, \omega, F^2)$ in 3.4. One notable
feature is that the sphere components of $F^2$ have to be discarded,
which resembles the definition of Thurston norm of $3-$manifolds.
 For a symplectic $4-$manifold
constructed via a positive genus fiber sum, the main result in
\cite{LY} can then be interpreted as a simple expression of its
Kodaira dimension in terms of the relative Kodaira dimensions of the
summands (Theorem \ref{fibersum}).

Another motivation comes from  the  paper \cite{DZ} by the second
author and J. Dorfmeister concerning the additivity of the Kodaira
dimensions for a $4-$dimensional Lefschetz fibration with singular
fibers. In that paper, the additivity is shown to hold in many
cases, while there is only a supadditivity relation in some cases.
It was speculated by the second author  whether this defect can be
remedied  if using appropriate relative Kodaira type invariants. For
this purpose,  we also introduce relative Kodaira dimension for a
$2-$manifold with a $\mathbb Q-$linear combination of points.
The well definedness is immediate in this case.
  We demonstrate that this
notion of relative Kodaira dimension can indeed be used to calculate
the Kodaira dimension of the total space for several kinds of
fibrations over surfaces with singular fibers.

The authors would like to thank Anar Akhmedov and Josef Dorfmeister
for very useful suggestions and discussions during the preparation
of the work, and Albert Marden for his interest. This research is
partially supported by NSF.

\tableofcontents

\section{Kodaira Dimensions and fiber bundles}
The goal of this section is to recall briefly the definitions of
various Kodaira dimensions mentioned in the introduction, and
establish the additivity for appropriate classes of fiber bundles.

\subsection{$\kappa^h$ for complex manifolds and $\kappa^t$ up to dimension $3$}

\subsubsection{The holomorphic Kodaira dimension $\kappa^h$}
Let  us first recall the original Kodaira dimension in complex
geometry.

\begin{definition} Suppose $(M, J)$ is a complex manifold of real dimension $2m$.
The holomorphic Kodaira dimension $\kappa^{h}(M,J)$ is defined as
follows:

\[
\kappa^{h}(M,J)=\left\{\begin{array}{ll}
-\infty &\hbox{ if $P_l(M,J)=0$ for all $l\ge 1$},\\
0& \hbox{ if $P_l(M,J)\in \{0,1\}$, but $\not\equiv 0$ for all $l\ge 1$},\\
k& \hbox{ if $P_l(M,J)\sim cl^k$; $c>0$}.\\
\end{array}\right.
\]

\end{definition}

Here  $P_l(M,J)$ is the $l-$th plurigenus of the  complex manifold
$(M, J)$ defined by $P_l(M,J)=h^0({\mathcal K}_J^{\otimes l})$, with
${\mathcal K}_J$ the canonical bundle of $(M,J)$.

\subsubsection{The topological Kodaira dimension $\kappa^t$ for manifolds up to dimension $3$}
As mentioned there are other situations where a similar notion can
be defined.  Let $M$ be a closed, smooth, oriented manifold.
  To begin with, we
make the following definition for  logical compatibility.

\begin{definition}
If $M=\emptyset$, then its Kodaira dimension   is defined to be
$-\infty$.
\end{definition}

 The only closed
connected $0-$dimensional manifold is a point, and the only closed
connected $1-$dimensional manifold is a circle.

\begin{definition} \label{01} If $M$ has dimension $0$ or $1$, then
its  Kodaira dimension $\kappa^t(M)$ is defined to be $0$.
\end{definition}

The $2-$dimensional Kodaira dimension is defined by the positivity
of the Euler class. We write $K=-e$.

\begin{definition}\label{2dim}
Suppose $M^2$ is a $2-$dimensional manifold with Euler class
$e(M^2)$. Write $K=-e(M^2)$ and define
\[ \kappa^t(M^2)=\left\{\begin{array}{cc}
-\infty & \hbox{ if $K<0$},\\
0& \hbox{ if $K=0$},\\
1& \hbox{ if $K> 0$}.
\end{array}\right.
\]
\end{definition}

It is easy to see that for any complex structure $J$ on $M^2$, $K$
is its canonical class, and  $\kappa^h(M^2, J)=\kappa^t(M^2)$.
 $\kappa^t(M^2)$ can be further interpreted from other viewpoints:
 symplectic structure
($K$ is also the symplectic canonical class), the Yamabe invariant,
geometric structures and etc.

Recall that the Yamabe invariant is defined in the following way
(\cite{Ko}, \cite{Sch}): \begin{equation}\label{yamabeinv}Y(M) =
\sup_{[g]\in \mathcal C}\inf_{g\in [g]}\int_M s_g
dV_g,\end{equation} where $g$ is a Riemannian metric on $M$, $s_g$
the scalar curvature of $g$, $[g]$ the conformal class of $g$, and
$\mathcal C$ the set of conformal classes on $M$.

A basic fact is  that $Y(M)>0$ if and only if $M$ admits a metric of
positive scalar curvature. Thus  $Y(M)$ is non-positive if $M$ does
not admit metrics of positive scalar curvature. Furthermore, in this
case, another basic fact is that $Y(M)$ is the supremum of the
scalar curvatures of all unit volume constant-scalar-curvature
metrics on $M$ (such metrics exist due to the resolution of the
Yamabe conjecture). It immediately follows that, in dimension two,
the sign of $Y(M^2)$ completely determines $\kappa^t(M^2)$.

We move on to dimension $3$. In this dimension the definition of the
Kodaira dimension in \cite{Z} by the second author is based on
geometric structures in the sense of Thurston. Divide the $8$
Thurston geometries into $3$ categories: $$\begin{array}{ll}
    -\infty: &\hbox{$S^3$ and $S^2 \times \mathbb{R}$};\cr
     0: &\hbox{$\mathbb{E}^3$, Nil and Sol};\cr
     1: &\hbox{$\mathbb{H}^2\times \mathbb{R}$, $\widetilde{SL_2(\mathbb{R})}$ and
    $\mathbb{H}^3$}.
\end{array}$$
Given a $3-$manifold $M^3$,  we decompose it first by a prime
decomposition and then further consider  a  toridal decomposition
for each prime summand, such that at the end each piece has a
geometric structure either in group $(1)$, $(2)$ or $(3)$ with
finite volume. The following definition was introduced in \cite{Z},
where the well definedness was also checked.

{\definition \label{t} For a $3-$dimensional manifolds $M^3$, we
define its Kodaira dimension  as follows:

\begin{enumerate}
    \item $\kappa^t(M^3)=-\infty$ if for any decomposition, each piece has  geometry type
    in category $-\infty$,
    \item $\kappa^t(M^3)=0$ if for any decomposition, we have at least a piece with geometry type
    in category $0$, but no piece has type in category $1$,
    \item $\kappa^t(M^3)=1$ if for any decomposition, we have at
    least one piece in category $1$.
\end{enumerate}
}

In this dimension, $Y(M^3)$ is also closely related to  geometric
structure of $M^3$, at least when $M^3$ is irreducible (see the
discussions in \cite{A} by Anderson).
However, as observed in \cite{Z}, the number $Y(M^3)$ does not
completely determine $\kappa^t(M^3)$. For $\Sigma_g \times S^1$, it
has vanishing Yamabe invariant if $g\geq 1$. But $\kappa^t(\Sigma_g
\times S^1)=0$ if $g=1$, $\kappa^t(\Sigma_g \times S^1)=1$ if $g\geq
2$. In this case, $\kappa^t$ is still determined by
\eqref{yamabeinv} if we  distinguish whether the supremum is
attainable by a metric. But this refinement of $Y(M^3)$ will still
not determine $\kappa^t$ since a Nil $3-$manifold like a non-trivial
$S^1-$bundle over $T^2$ has  Yamabe invariant $0$ which is not
attainable by any metric.

Notice that here we use $\kappa^t$ to denote the Kodaira dimension
for smooth manifolds in dimensions $0, 1, 2, 3$. Here $t$ stands for
{\it topological}, because in these dimensions homeomorphic
manifolds are actually diffeomorphic.

 For a possibly disconnected manifold, we define its
Kodaira dimension to be the maximum of that of its components. In
summary, we have defined the Kodaira dimension for all the closed,
oriented  manifolds with dimension less than $4$.

\subsection{$\kappa^s$ for symplectic $4-$manifolds} In \cite{L1},  the first
author systematically investigated the notion of symplectic Kodaira
dimension for  symplectic $4-$manifolds. To define it we need to
first recall the notion of minimality.

\subsubsection{Minimality in dimension
$4$}
\begin{definition} Let ${\mathcal E}_M$ be the set of cohomology
classes whose Poincar\'e dual are represented by smoothly embedded
spheres of self-intersection $-1$. $M$ is said to be (smoothly)
minimal if ${\mathcal E}_M$ is the empty set.
 \end{definition}

Equivalently, $M$ is minimal if it is not the connected sum of
another manifold  with $\overline{\mathbb {CP}^2}$. We say that $N$
is a minimal model of $M$ if $N$ is  minimal and $M$ is the
connected sum of $N$ and a number of $\overline{\mathbb {CP}^2}$.

We  also recall the notion of minimality for  $(M,\omega)$.
$(M,\omega)$  is said to be (symplectically) minimal if ${\mathcal
E}_{\omega}$ is the empty set, where
$${\mathcal E}_{\omega}=\{E\in {\mathcal
E}_M|\hbox{ $E$ is represented by an embedded $\omega-$symplectic
sphere}\}.$$ A basic fact proved using SW theory (\cite{T},
\cite{LL1}, \cite{L3}) is:
  ${\mathcal E}_{\omega}$ is
empty if and only if ${\mathcal E}_M$ is empty. In other words, $(M,
\omega)$ is symplectically minimal if and only if $M$ is smoothly
minimal.

\subsubsection{Definitions}
\begin{definition}\label{sym Kod'}
For a minimal symplectic $4-$manifold $(M^4,\omega)$ with symplectic
canonical class $K_{\omega}$,   the Kodaira dimension of
$(M^4,\omega)$
 is defined in the following way:

$$
\kappa^s(M^4,\omega)=\begin{cases} \begin{array}{lll}
-\infty & \hbox{ if $K_{\omega}\cdot [\omega]<0$ or} & K_{\omega}\cdot K_{\omega}<0,\\
0& \hbox{ if $K_{\omega}\cdot [\omega]=0$ and} & K_{\omega}\cdot K_{\omega}=0,\\
1& \hbox{ if $K_{\omega}\cdot [\omega]> 0$ and} & K_{\omega}\cdot K_{\omega}=0,\\
2& \hbox{ if $K_{\omega}\cdot [\omega]>0$ and} & K_{\omega}\cdot K_{\omega}>0.\\
\end{array}
\end{cases}
$$

The Kodaira dimension of a non-minimal manifold is defined to be
that of any of its minimal models.
\end{definition}

Here $K_{\omega}$ is defined as the first Chern class of the
cotangent bundle for any almost complex structure compatible with
$\omega$.

 We here offer an interpretation of $\kappa^s$ which relates it to
the  $2-$dimensional $\kappa^t$.
 Define the (symplectic) Kodaira dimension for a number $k$  (or equivalently, a top
dimensional cohomology class of a closed oriented manifold) in the
following way:

\begin{equation}\label{kappa(r)}\kappa^s(r)=\begin{cases} \begin{array}{lll}
-\infty & \hbox{ if $r<0$},\\
0& \hbox{ if $r=0$},\\
1& \hbox{ if $r> 0$}.\\
\end{array}
\end{cases}
\end{equation}
%
Then for a $2-$dimensional manifold $F^2$, we have
$$\kappa^t(F^2)=\kappa^s(-e(F^2))=\kappa^s(-\chi(F^2)),$$
where $\chi$ denotes the Euler characteristic.
 Furthermore, for a $4-$dimensional minimal symplectic manifold
 $(M^4, \omega)$,
\begin{equation} \label{add} \kappa^s(M^4,
\omega)=\kappa^s(K_{\omega}^2)+\kappa^s(K_{\omega}\cdot [\omega]).
\end{equation}

 We further  make a couple of easy observations based on \eqref{add}.

\begin{lemma}\label{ss'}
Let $(M^4, \omega)$ be a minimal symplectic manifold. If
$K_{\omega}^2< 0$, then $\kappa^s(M^4,
\omega)=\kappa^s(K_{\omega}^2)=-\infty$. If $K_{\omega}^2\geq 0$,
then
\begin{equation}
\kappa^s(\kappa^s(M^4, \omega))=\kappa^s(K_{\omega}\cdot [\omega]).
\end{equation}
\end{lemma}

 Due to the properties of $\kappa^s$ listed in \cite{L1}, such as
   the diffeomorphism invariance of $\kappa^s$, we can yet regard
   $\kappa^s$ as an invariant of a large class of smooth $4-$manifolds in the following
   way.

\begin{definition} \label{sym Kod} Suppose that $M^4$ is a $4-$dimensional closed,
oriented  manifold admitting symplectic structures (compatible with
the orientation). $M^4$ is said to have symplectic Kodaira dimension
$\kappa^s=-\infty$ if $M^4$ is rational or ruled.

Otherwise, first suppose  that $M^4$ is smoothly minimal. Then the
Kodaira dimension $\kappa^s$ of $M^4$ is defined   as follows:

\[
\kappa^s(M^4)=\kappa^s(M^4,\omega)=\left\{\begin{array}{ll}
0& \hbox{ if $K_\omega$ is torsion},\\
1& \hbox{ if $K_\omega$ is non-torsion
but $K_\omega^2=0$},\\
2& \hbox{ if $K_\omega^2>0$}.\\
\end{array}\right.
\]
Here $\omega$ is any symplectic form on $M^4$ compatible with the
orientation.

For a general $M^4$, $\kappa^s(M^4)$ is defined to be
$\kappa^s(N^4)$, where $N^4$ is a smooth minimal model of $M^4$.
\end{definition}

Here a rational $4-$manifold is $S^2\times S^2$, or ${\mathbb
{CP}^2}\# k{\overline{\mathbb {CP}^2}}$ for some non-negative
integer $k$. A ruled $4-$manifold is the connected sum of a number
of (possibly zero) $\overline {\mathbb {CP}^2}$ with an $S^2-$bundle
over a Riemann surface.

It was verified in \cite{DZ} that $\kappa^s=\kappa^h$ whenever both
are defined. In fact it was shown earlier in \cite{FM} that
$\kappa^h(M^4, J)$ (even the plurigenera) only depends on the
oriented diffeomorphism type of $M^4$.

 LeBrun
(\cite{LeB2}) calculated $Y(M^4)$ when $M^4$ admits a  K\"ahler
structure, from which he concluded that \eqref{yamabeinv} completely
determines $\kappa^h$. As $\kappa^s=\kappa^h$ for a K\"ahler
surface,
 we can rephrase LeBrun's calculation
 in the following way: If $M^4$ admits a K\"ahler structure,
 then
\begin{equation}\label{lebrun}  \kappa^s(M^4)=\left\{\begin{array}{ll}
-\infty & \hbox{ if $Y(M^4)>0$},\\
0& \hbox{ if $Y(M^4)=0$ and $0$ is attainable by a metric},\\
1& \hbox{ if $Y(M^4)=0$ and $0$ is not attainable},\\
2& \hbox{ if $Y(M^4)<0$}.
\end{array}\right.
\end{equation}

However, \eqref{lebrun} does not determine $\kappa^s(M^4)$ for all
symplectic $M^4$: All $T^2-$bundle over $T^2$ have $\kappa^s=0$ (see
\cite{L1}) while most of them do not have any zero scalar curvature
metrics. But the question of LeBrun in \cite{LeB1} still makes
sense:
 if $M^4$ admits a symplectic structure and $Y(M^4)<0$, is $\kappa^s(M^4)=2$?

A related question is whether we can extend $\kappa^s$ and
$\kappa^h$ to $\kappa^d$ for all smooth $4-$manifolds (here $d$
standing for diffeomorphic).


\subsubsection{Higher Dimension}\label{5D}
 In higher dimension,  Kodaira dimension is only defined for
complex manifolds. And $\kappa^h$ is known not to be a
diffeomorphism invariant. Here is a specific example following
\cite{R}.


Consider a Fano surface $(M^4=\mathbb {CP}^2\sharp
5\overline{\mathbb {CP}^2}, J_1)$, and  a complex surface $(N^4,
J_2) $ of general type homeomorphic to $M^4$ as constructed by J.
Park et al(\cite{PPS}). Then $(M^4, J_1) \times (T^2, j)$ and $(N^4,
J_2)\times (T^2, j)$ are complex manifolds, and they are
diffeomorphic by the $s-$cobordism theorem (as they are
$h-$cobordant and their Whitehead groups vanish, for details see
\cite{R}). However, their complex Kodaira dimensions are different
due to the additivity property of $\kappa^h$ for a product.
Similarly,   the pair of diffeomorphic $5-$manifolds $M^4\times S^1$
and $N^4\times S^1$ tells us that, there is no smoothly invariant
definition of Kodaira dimension in dimension $5$  if we require the
very natural
 additivity for a product manifold.

Thus we can only expect to have a notion of Kodaira dimension for
manifolds with some structures such as complex structures or
symplectic structures (for the latter case see the proposal in
\cite{LR}).

\subsection{Additivity for a fiber bundle}
We discuss additivity of  the Kodaira dimensions $\kappa^h,
\kappa^t, \kappa^s$  for appropriate classes of fiber bundles.
\subsubsection{Additivity for $\kappa^h$}
Let us start with the holomorphic Kodaira dimension $\kappa^h$. A
classical theorem says that the additivity holds for a holomorphic
fiber bundle (see Theorem 15.1 in \cite{Ue} for example).
Especially, $\kappa^h$ is covering invariant.

\subsubsection{Covering invariance} We start with  fiber bundles with
$0-$dimensional fibers, namely, unramified coverings.

\begin{prop}\label{st}
The Kodaira dimensions $ \kappa^t, \kappa^s$ are  covering
invariant.
\end{prop}

\begin{proof}

For $0-$ and $1-$manifolds, it is obvious. For $2-$dimensional
manifolds it follows from the fact that $\chi(\widetilde
M)=n\chi(M)$ if $f:\widetilde M\to M$ is a degree $n$ covering map.
For $3-$dimensional manifolds, it is more or less clear from the
definition and was verified in \cite{Z}.

It remains to check $\kappa^s$. First of all,  if $f:\tilde M^4\to
M^4$ is a covering map and $\omega$ is a symplectic form on $M^4$,
then $f^*\omega$ is a symplectic form on $\tilde M^4$, and thus
$\kappa^s(\tilde M^4)$ is defined.

One characterization of $\kappa^s=-\infty$ manifolds is the
existence of an embedded symplectic sphere with non-negative
self-intersection (\cite{Mcjams}, \cite{LL1}). Suppose
$\kappa^s(M^4, \omega)=-\infty$ and $F\subset(M^4, \omega)$ is an
embedded symplectic sphere with $[F]^2\geq 0$. As $F$ is simply
connected, $f^{-1}(F)\subset (\tilde M^4, f^*\omega)$ consists of
$l=\deg(f)$ symplectic spheres, each still with self-intersection
$[F]^2$. Thus $\kappa^s(\tilde M^4, f^*\omega)=-\infty=\kappa^s(M^4,
\omega)$.

In fact, we can easily enumerate all the coverings in the case
$\kappa^s(M^4, \omega)=\infty$. Assume first that $(M^4, \omega)$ is
minimal. Then $M^4=\mathbb {CP}^2, S^2\times S^2$ or an $S^2-$bundle
over $\Sigma_{h\geq 1}$. If $M^4=\mathbb {CP}^2, S^2\times S^2$,
then so is $\tilde M^4$. If $M^4$ is an $S^2-$bundle over
$\Sigma_{h\geq 1}$, then $\tilde M^4$ is  an $S^2-$bundle over
$\Sigma_{h'\geq 1}$, induced by a covering $\Sigma_{h'}\to
\Sigma_h$. In particular,  $(\tilde M^4, f^*\omega)$ is still
minimal.

For a non-minimal $(M^4, \omega)$ we have the following general
observation:
 When $(M^4, \omega)$ is a blow up of $(N^4, \tau)$ around a symplectic
ball $B^4$, we observe that  $(\tilde M^4, f^*{\omega})$ is the blow
up of $(\tilde N^4, g^*\tau)$, where $\tilde N^4$ is obtained by
gluing $\deg(f)$ copies of $B^4$ to $\tilde f^{-1}(N^4-B^4)$, and
$g:\tilde N^4 \to N^4$ is the obvious covering map.

To prove \eqref{ss} when $\kappa^s(M, \omega)\geq 0$, we need the
following fact.

\begin{lemma}\label{minmin}
 $(\tilde M^4, f^*\omega)$ is minimal if and only if $(M^4,
\omega)$ is minimal.
\end{lemma}

Let us first assume Lemma \ref{minmin}.
 Using the fact
$K_{f^*\omega}=f^*K_{\omega}$, we have
$$K_{f^*\omega}\cdot [f^*\omega]=\deg(f) \, K_{\omega}\cdot
[\omega], \quad K_{f^*\omega}\cdot K_{f^*\omega}=\deg (f)\,
K_{\omega}\cdot K_{\omega}.$$ Together with Lemma \ref{minmin}, it
follows that \begin{equation} \label{ss} \kappa^s(\tilde M^4,
f^*\omega)=\kappa^s(M^4, \omega)\end{equation} when $(M^4, \omega)$
is minimal.

Now,  \eqref{ss} for a general $(M^4, \omega)$ is a consequence of
the  observation made before Lemma \ref{minmin}.

It only remains to prove Lemma \ref{minmin}.
\begin{proof}
 Suppose $(M^4, \omega)$
is not minimal. Then there is a symplectic $-1$ sphere $S$ in $(M^4,
\omega)$. As $S$ is simply connected, $f^{-1}(S)\subset (\tilde M^4,
f^*\omega)$ consists of $l=\deg(f)$ symplectic spheres, each still
with self-intersection $-1$.

Suppose $(M^4, \omega)$ is minimal. We want to prove that $(\tilde
M^4, f^*\omega)$ is also minimal. The case $\kappa^s(M^4,
\omega)=-\infty$ is already settled. When $\kappa^s(M^4, \omega)\geq
0$, for a generic $\omega-$compatible almost complex structure $J$,
according to Taubes (\cite{T}), $K_{\omega}$ is represented by a
$J-$holomorphic submanifold $C$, possibly disconnected and empty,
but without sphere components. Let $\tilde J=f^*J$. Then  $\tilde
C=f^{-1}(C)$ is  a $\tilde J-$holomorphic submanifold of $(\tilde
M^4, \tilde J)$ representing $K_{f^*\omega}$. Notice that $\tilde C$
still has no sphere components. If $(\tilde M^4, f^*\omega)$ is not
minimal and $\tilde E\in \mathcal E_{f^*\omega}$, then there is a
$\tilde J-$holomorphic curve $V$ in the class of $\tilde E$. The
curve $V$ could be singular and reducible, but every component of
$V$ has to have genus $0$. In particular,   $V$ and $\tilde C$ have
no common components. By the positivity of intersection of distinct
irreducible pseudo-holomorphic curves, $[V]\cdot [\tilde C]\geq 0$.
But this contradicts to  $K_{f^*\omega}\cdot \tilde E=-1$.
\end{proof}


\end{proof}




We note that the notion of Kodaira dimension does not depend on the
orientation of the manifold in dimension at most $3$. So we could
extend it to a connected non-orientable manifold up to dimension $3$
using its unique orientable covering.

Let us mention that the sign of the Yamabe invariant $Y(M)$ is
generally not a covering invariant.  In dimension $4$, LeBrun in
\cite{LeB} constructed a reducible non-symplectic manifold $M^4$
with $Y(M^4)<0$, whose universal covering is $k \mathbb {CP}^2\sharp
l\overline{\mathbb {CP}^2}$, hence having positive Yamabe invariant.
This example also shows that the condition that $M$ admits
symplectic structures in Lemma \ref{minmin} is necessary.


\subsubsection{Bundles in dimensions at most four} \label{3D}

The following is essentially contained in \cite{Z}.
\begin{prop} \label{21}
$\kappa^t$ is additive for any fiber bundle in dimension at most
$3$.
\end{prop}

The statement is obvious when the base is  $0$ dimensional. When
fibers are $0-$dimensional, it is just Proposition \ref{st}.

It is also obvious that $\kappa^t$ is additive for any circle bundle
when the total space is of dimension $1$ or $2$, even if the bundle
is not orientable.

There are two kinds of bundles in dimension $3$: circle bundles over
surface and surface bundles over circle. In both cases the
additivity of $\kappa^t$ is shown in \cite{Z}. Circle bundles are
special Seifert fiber spaces.  See \ref{seifert} for related
discussions.

In dimension $4$ we have the following additivity results.

\begin{prop}\label{31}
Suppose $M^3\times S^1$ has a symplectic (complex) structure, then
$\kappa^{s(h)}(M^3\times
S^1)=\kappa^t(S^1)+\kappa^t(M^3)=\kappa^t(M^3)$.
\end{prop}

\begin{prop} \label{bundle}
Suppose $M^4$ is a surface bundle over surface and it has a
symplectic (complex) structure, then
$$\kappa^{s(h)}(M^4)=\kappa^t(base)+\kappa^t(fiber).$$
\end{prop}

Proposition \ref{31} is contained in \cite{Z}.  For symplectic case,
it depends on  the resolution of the Taubes conjecture by  Friedl
and Vidussi (\cite{FV}). For complex case it depends on \cite{E}.
Hopefully, we can generalize it to  $M^3$ bundles over $S^1$ or
$S^1$ bundles  over $M^3$.

Proposition \ref{bundle} is  established in \cite{DZ} when the base
surface has  positive genus. When the base is $S^2$, the total space
is either a ruled manifold which is symplectic and complex and has
$\kappa^s=\kappa^h=-\infty$, or a Hopf surface which is complex and
has $\kappa^h=-\infty$ (the latter case occurs when the fiber is
$T^2$ and homologically trivial).




\section{Embedded symplectic surfaces and relative Kod. dim. in dim. 4}

 In this section $M$ denotes a smooth, oriented, closed and
connected $4-$manifold, $\omega$ denotes a symplectic form on $M$
compatible with the orientation.

We often identify a degree 2 homology class with its Poincar\'e
dual, and vice versa. We  denote by $\cdot$ the  pairing between a
degree 2 homology class and a degree 2 cohomology class, the
intersection product of two degree 2 homology classes, as well as
the cup product of two degree 2 cohomology classes.
\subsection{Embedded symplectic surfaces and
maximality}\label{minmax}

\subsubsection{Embedded symplectic surfaces}
\begin{definition}\label{formal} Suppose
 $F\subset (M, \omega)$ is a symplectically embedded
surface (possibly disconnected). Its  genus is defined by
\begin{equation}\label{adjunction}
2g(F)-2=K_{\omega}\cdot [F]+[F]^2.
\end{equation}
More generally, for a class $e\in H_2(M)$ we use \eqref{adjunction}
to define the $\omega-$genus $g_{\omega}(e)$ of $e$.
\end{definition}

If $F$ is connected, \eqref{adjunction} is just the adjunction
formula, and thus the (formal) genus $g(F)$ defined by
\eqref{adjunction} is just the usual genus of $F$. Observe also that
if $F=\sqcup F_i$ with connected components $F_i$, then
\begin{equation}
2g(F)-2=K_{\omega}\cdot [F]+[F]^2=\sum(K_{\omega}\cdot
[F_i]+[F_i]^2)=\sum(2g(F_i)-2).
\end{equation}

In particular, we have

\begin{lemma} \label{disc} Suppose $F=\sqcup F_i$ with connected components $F_i$.

(i) If each $F_i$ has positive genus, then $g(F)\geq 1$.

(ii) If $F$ admits  a degree $d$ map to a connected surface of genus
$h$, then $g(F)\geq dh-d+1$.
\end{lemma}

Recall that a degree 2 class is called GW stable in \cite{L6} if
certain GW invariant of this class is nonzero. The next lemma  is
well-known (cf. \cite {LL}, \cite {Mc2}).

\begin{lemma} \label{special} The following classes are GW stable
classes. \begin{itemize}

\item  The class of an embedded
symplectic sphere with non-negative self intersection.

\item Any symplectic $-1$ class $E\in {\mathcal E}_{M, \omega}$.

\item $K_{\omega}-E_1-\cdots-E_p$ with $E_i\ne E_j \in {\mathcal E}_{M, \omega}$
when $\kappa^s(M, \omega)\geq 0$ and $b^+>1$.

\item $2K_{\omega}-E_1-\cdots-E_p$ with $E_i\ne E_j \in {\mathcal E}_{M, \omega}$
when $\kappa^s(M, \omega)\geq 0$ and $b^+=1$.

\item $-K_{\omega}$ when $(M, \omega)$ is minimal and
$K_{\omega}^2\geq 0$.
\end{itemize}
\end{lemma}

The following simple fact was observed in \cite{L6}.

\begin{lemma} \label{stable}
If $\alpha\in H_2(M;\mathbb Z)$ with $\alpha^2\geq 0$ is represented
by an embedded symplectic surface, then $\alpha$ pairs
non-negatively with any GW stable class.
\end{lemma}

Finally, for a possibly disconnected embedded  surface $F=\sqcup
F_i$ in $M$ with connected components $F_i$,  let $F^+$ be the union
of $F_i^+$, where
\[
{F_i}^+=\left\{\begin{array}{ll}
F_i & \hbox{ if $\kappa^t({F_i})\neq -\infty$},\\
\emptyset & \hbox{ if $\kappa^t({F_i})= -\infty$}.\\
\end{array}\right.
\]
\subsubsection{$\kappa^s(K_{\omega}\cdot [F])$ and $\kappa^s(K_{\omega}\cdot [\omega])$}

\begin{lemma} \label{F=omega}
Let $(M, \omega)$ be a minimal symplectic manifold with
$K_{\omega}^2\geq 0$. Suppose $S$ is a symplectic surface with $S^2>
0$. We further suppose that there is  a relatively minimal Lefschetz
fibration on $\tilde M=M\sharp k\overline{\mathbb {CP}^2}$ such that
the class of a fiber $\tilde S$ satisfies  $\pi_*[\tilde S]=[S]$,
where $\pi_*:H_2(\tilde M)\to H_2(M)$ is the natural homomorphism.
Then
\begin{equation}\label{sss}\kappa^s(K_{\omega}\cdot [S])=\kappa^s(\kappa^s(M,
\omega)).\end{equation}


\end{lemma}

\begin{proof}
First of all, under the assumption that $(M, \omega)$ is minimal and
$K_{\omega}^2\geq 0$, by Lemma \ref{special}, $K_{\omega}$ or
$2K_{\omega}$ is a GW stable class if $\kappa^s(M, \omega)\geq 0$,
and $-K_{\omega}$ is a GW stable class if $\kappa^s(M,
\omega)=-\infty$. By Lemma \ref{stable}, $K_{\omega}\cdot [S]\geq 0$
if $\kappa^s(M, \omega)\geq 0$, and $K_{\omega}\cdot [S]\leq 0$ if
$\kappa^s(M, \omega)=-\infty$.

Thus, if $K_{\omega}\cdot [S]<0$, we must have $\kappa^s(M,
\omega)=-\infty$. Conversely, if $\kappa^s(M, \omega)=-\infty$,
 since $[S]^2>0$ and $K_{\omega}\cdot [S]\leq 0$, by the light cone lemma, we have $K_{\omega}\cdot
[S]<0$.

If $\kappa^s(M, \omega)=0$, then $K_{\omega}$ is a torsion class. So
$K_{\omega}\cdot [S]=0$ in this case.



To prove \eqref{sss}, what remains to show is that if $\kappa^s(M,
\omega)\geq 1$, then $K_{\omega}\cdot [S]>0$.
It is here that we need the assumption that $[S]$ lifts to the fiber
class $[\tilde S]$ of a Lefschetz fibration on $(\tilde M=M\sharp
k\overline{\mathbb {CP}^2}, \tilde \omega)$. Notice that since
$[S]^2>0$, $[S]$ itself cannot be the fiber class, thus we must have
$k>0$. Notice also that $K_{\omega}\cdot [S]=\pi^*K_{\omega}\cdot
[\tilde S]$.

In this case,  $\kappa^s(\tilde M, \tilde \omega)\geq 1$. Then
$\pi^*K_{\omega}$, or $\pi^*(2K_{\omega})$ in the case $b^+=1$, is
still a GW stable class in the blow up $(\tilde M, \tilde \omega)$.
Here $\pi^*:H^2(M)\to H^2(\tilde M)$ is the natural inclusion.
Choose an almost complex structure $J$ on $\tilde M$ making the
Lefschetz fibration $J-$holomorphic. What can a $J-$holomorphic
representative of $\pi^*K_{\omega}$ ($2\pi^*K_{\omega}$) be? If it
is in a fiber or a union of several fibers,  then its square is at
most $0$, and its square is $0$ only if it is a union of fibers.
Thus if $\kappa^s(M, \omega) =2$, this is impossible. If
$\kappa^s(M, \omega) = 1$, it still violates the fact that the
intersection number of $\pi^*K_{\omega}$ with any $-1$ class of
$(\tilde M, \tilde \omega)$ is $0$. Thus $K_{\omega}$ must have a
multi-section component. This shows that $\pi^* K_{\omega} \cdot
[\tilde S] > 0$, and hence $K_{\omega}\cdot [S]>0$.



\end{proof}
Following from Lemma \ref{ss'}, we have
\begin{cor}\label{sss'}
Under the assumption of Lemma \ref{F=omega}, \begin{equation}
\kappa^s(K_{\omega}\cdot [S])=\kappa^s(K_{\omega}\cdot [\omega]).
\end{equation}
\end{cor}

Any member of Lefschetz pencil on $(M, \omega)$ satisfies the
assumption of Lemma \ref{F=omega}. In this case,  Gompf (\cite{G})
showed that there is a symplectic form $\tau$ on $M$ in the positive
ray of $[S]$. It would be interesting to see  whether this remains
to be true for any $S$ as in Lemma \ref{F=omega}.


Suppose $E_i$ are the classes of symplectic $-1$ spheres in $(\tilde
M, \tilde \omega)$ that are blown down to obtain $(M, \omega)$.
Since $[\tilde S]^2=0$,
\begin{equation}
2g(\tilde S)-2= K_{\omega'}\cdot [\tilde S]=(\pi^* K_{\omega}+\sum
E_i)(\iota_*[S]-\sum ([\tilde S]\cdot E_i)E_i),
\end{equation}
where $\iota_*:H_2(M)\to H_2(\tilde M)$ is the natural inclusion.
Thus we can express $\kappa^s(K_{\omega}\cdot[\omega])$ in terms of
$g(\tilde S)$ and $c_i=[\tilde S]\cdot E_i$,
\begin{equation}\label{final}
\kappa^s(K_{\omega}\cdot[\omega])=\kappa^s(2g(\tilde S)-2-\sum c_i).
\end{equation}




\subsubsection{Maximal surfaces}
\begin{definition}\label{max'}
Suppose $F\subset (M, \omega)$ is a symplectically embedded surface
without sphere components. $F$ is called maximal if $[ F]\cdot E\ne
0$ for any $E\in \mathcal E_{\omega}$.

For a general embedded symplectic surface $F$, it is called maximal
if $F^+$ is maximal.
\end{definition}

Any member of a relatively minimal Lefschetz pencil or a fiber of a
relatively minimal Lefschetz fibration is maximal. Notice that if
$F^+=\emptyset$, then $F$ is maximal if and only if $(M, \omega)$ is
minimal.

Let $F_i$ be the connected components of an embedded symplectic
surface $F$. Because the $F_i$ are disjoint  and embedded symplectic
surfaces, we can choose an almost complex structure $J$ to make each
$F_i$ $J-$holomorphic.

\begin{claim}\label{claim} Suppose the genus of each $F_i$ is positive. Then for any $E\in
\mathcal E_{\omega}$, we can further assume that  $J$ is chosen such
that both $F$ and an embedded representative of $E$  are
$J-$holomorphic.
\end{claim}

\begin{proof}
This can be done, for example, by Proposition $4.1$ in \cite{OO}. We
recall the argument here:  Without loss of generality, we assume
that $F$ is connected.

First, we choose a $J_0$ such that $F$ is $J_0$ holomorphic. We can
assume that $J_0$ is generic outside a small neighborhood $U$ of $F$
so that any simple $J_0$ holomorphic curve which are not contained
in $U$ are transversal. Suppose $E$ and $[F]$ can not be represented
by $J-$holomorphic curves simultaneously. Choose a sequence of $J_n$
converging to $J_0$ such that $E$ is represented by the embedded
$J_n-$holomorphic $-1-$curve $E_n$ for all $n$. Then $E_n$ converges
to the image of a stable map $\sum m_iB_i$, where $B_i$'s are
simple. Here $\sum m_i>1$.

Now we show that one of $\{B_i\}$ is contained in $U$. If not, they
are transversal by our genericness assumption of $J_0$. Hence, for
$n$ large enough, $B_i$ deform to $J_n-$holomorphic $B_i'$. Thus
$E_n$ and $\sum m_iB_i'$ are both $J_n-$holomorphic curves
representing class $E$. If $E_n$ does not appear in $\{B_i\}$, then
$-1=E^2=[E_n]\cdot \sum m_i[B_i'] \ge 0$, a contradiction. So there
is an $i$ so that $B_i=E_n$. It never happens because symplectic
area only depends on the homology class and $\sum m_i>1$.

Now, note that each component of the stable map above is of genus
$0$ and at least one of them is possibly a multiple cover of $F$,
whose genus is positive. This is impossible.
\end{proof}

Thus we can conclude

\begin{lemma} \label{max} Suppose $F\subset (M, \omega)$ is a symplectically  embedded surface
without sphere components.

$\bullet$ $F\neq \emptyset$ is  maximal in $(M, \omega)$ in the
sense of Definition \ref{max'} if and only if  $[F]\cdot E>0$ for
any $E\in \mathcal E_{\omega}$.

$\bullet$ If  $[F]\cdot E=0$ for some $E\in \mathcal E_{\omega}$,
then we can blow down a symplectic sphere in the class $E$ which is
disjoint from $F$.
\end{lemma}

Here is another useful consequence of Claim \ref{claim}.

\begin{lemma} \label{blowdown} Suppose $(N, \sigma)$ is obtained from $(M, \omega)$
by blowing down a finite set of disjoint symplectic $-1$ spheres in
the classes $E_i$. Then for any embedded symplectic surface
$F\subset(M, \omega)$, possibly disconnected but with each component
positive genus, there is an embedded symplectic surface $F'\subset
(N, \sigma)$ with each component positive genus such that
\begin{equation}
[F]=\iota_*[F']-\sum([F]\cdot E_i) E_i.
\end{equation}
Here $\iota_*:H_2(N)\to H_2(M)$ is the natural inclusion.
\end{lemma}
\begin{proof}  To apply Lemma \ref{max} we blow down the $-1$ classes successively. We choose an
$\omega-$tamed almost complex structure  $J$ as in the proof of
Claim  \ref{claim} such that $E_1$ is represented by an embedded
$J-$holomorphic sphere $S_1$,  and  $F$ is $J-$holomorphic. By a
small isotopy of $F$, we can further assume that $F$ is symplectic
and intersects $S_1$ transversally and non-negatively. We can then
perform blow down such that $F$ becomes an immersed symplectic
surface  with only positive nodal points and still without sphere
components. Here it is convenient  to view blowing down as fiber summing with the pair $\mathbb CP^2$ along a line,
and from this point of view, the immersed symplectic surface is obtained from 
 Gompf's pairwise fiber sum construction (\cite{Gompf}). 
Observe that Claim \ref{claim} actually generalizes to a
positively immersed symplectic surface as it still can be made
pseudo-holomorphic. Then we repeat this process to finally obtain an
immersed symplectic surface $F_{red}$ in $(N, \sigma)$ with only
positive nodal points and still without sphere components. By
Corollary $3.4$ in \cite{LU}, we can perturb it to an embedded
symplectic surface $F'$. Notice that, if $c_i=F\cdot E_i$, then,
$$[F]=\iota_* [F_{red}]-\sum_1^k c_iE_i=\iota_*[F']-\sum_1^k c_iE_i.$$
\end{proof}

\subsection{The adjoint class $K_{\omega}+[F]$}
The following definition was introduced in  \cite{LY} for a
connected surface.

\begin{definition}\label{adjoint} Let $F$ be an embedded symplectic surface
in $(M, \omega)$ with each component positive genus.
 The adjoint class of $F$ is defined as $K_{\omega}+[F]$.
\begin{itemize}
 \item $F$ is
called maximal if  for any symplectic $-1$ class $E$,
$$(K_{\omega}+[F])\cdot E\geq 0. $$

\item  $F$ is called special if
$(K_{\omega}+[F])^2=0.$

\item $F$ is called distinguished if $K_{\omega}+[F]$ is rationally
trivial.
\end{itemize}

\end{definition}
 As $K_{\omega}\cdot E=-1$ for any $E\in
\mathcal E_M$, by Lemma \ref{max}, the two notions of maximality in
Definitions \ref {max'} and \ref{adjoint} coincide when $F$ is an
embedded symplectic surface without sphere components.

In this subsection we assume that $F$ is a maximal  symplectic
surface in $(M, \omega)$ with each component positive genus.

\subsubsection{$\kappa^s((K_{\omega}+[F])^2)$}
We now discuss the sign of $(K_{\omega}+[F])^2$ for a maximal
symplectic surface $F$, in other words, we calculate
$\kappa^s((K_{\omega}+[F])^2)$.


 \begin{prop}\label{crucial}
 Suppose $\kappa^s(M, \omega)\geq 0$, $F=\sqcup F_i$ is a maximal symplectic surface and each $F_i$ of positive genus.
Then we have
 \begin{equation}\label{KF<0}(K_{\omega}+[F])^2\geq 0.
 \end{equation}
 \end{prop}
\begin{proof}
Notice that when $F$ is connected the statement   is contained  in
\cite{LY}. We point out however when $[F]^2<0$ some further
arguments, e.g. those in the appendix in \cite{DL}, are needed to
complete the proof there. We here offer an alternative argument for
this more general (possibly disconnected) situation.

 Let us rewrite
\begin{equation}\label{3}
(K_{\omega}+[F])^2=K_{\omega}^2+K_{\omega}\cdot [F]+(K_{\omega}\cdot
[F]+[F]^2)
\end{equation}
as a sum of three terms.

 First
let us suppose that $F$ is connected. By \eqref{adjunction}
 the last term is non-negative as $g(F)\geq
1$. Let us argue that
\begin{equation}\label{KF} K_{\omega}\cdot [F]\ge
0.\end{equation} When $[F]^2\ge 0$,  it is due to  lemmas
\ref{stable} and \ref{special}; when $[F]^2<0$, because $g(F)\ge 1$,
it is  due to the adjunction formula \eqref{adjunction}.

If we further assume that $(M, \omega)$ is minimal, then
$K_{\omega}^2\geq 0$ as well.  Thus we can conclude that
 (\ref{KF<0}) holds when  $F$ is connected and $(M, \omega)$ is minimal.

For a disconnected symplectic surface $F=\sqcup F_i$, as each
connected component $F_i$ has positive genus, we still have by the
adjunction formula,
\begin{equation}
(K_{\omega}\cdot [F]+[F]^2)=\sum (K_{\omega}\cdot [F_i]+[F_i]^2)\geq
0.
\end{equation}
Moreover, $K_{\omega}\cdot [F_i]\geq 0$, so
\begin{equation}\label{KF'<0}
K_{\omega}\cdot [F]=\sum K_{\omega}\cdot [F_i]\geq 0.
\end{equation}
Thus if $(M, \omega)$ is minimal, all  three terms in \eqref{3} are
still non-negative.

In summary we have shown that (\ref{KF<0}) holds when $(M, \omega)$
is minimal.

Now  we assume that  $(M, \omega)$ is non-minimal and $\mathcal
E_{\omega}=\{E_i\}$. Then $K_{\omega}^2$ could be negative. However,
as the 3rd term in \eqref{3} is always non-negative,  it suffices to
prove that the sum of the 1st and the 2nd terms
\begin{equation} \label{sum} K_{\omega}^2+K_{\omega}\cdot [F]
\end{equation} in \eqref{3} is non-negative.

 Let $(N, \sigma)$ be the minimal model of $(M,
\omega)$ and $K_{\sigma}$ be its symplectic canonical class. Then
 $$K_{\omega}=\pi^*K_{\sigma}+\sum_1^k E_i,$$
where $\pi^*:H^2(N)\to H^2(M)$ is the natural inclusion.
 By Lemma
\ref{blowdown}, there is an embedded symplectic surface $F'\subset
(N, \sigma)$ such that \begin{equation} [F]=\iota_*[F']-\sum c_i
E_i,\quad c_i=F\cdot E_i.\end{equation} As argued above, we have
\begin{equation} \label{sum'} K_{\sigma}^2+K_{\sigma}\cdot [F']\geq
0.\end{equation}
 The contribution of $E_i$ to $
K_{\omega}^2$ is $-k$, to $K\cdot[F]$ is $\sum_1^kc_i$. Because $F$
is maximal,  $c_i\ge 1$. Thus the difference of \eqref{sum} and
\eqref{sum'} is non-negative.

\end{proof}

From the arguments above, it is easy to determine when
$(K_{\omega}+[F])^2=0$.

 \begin{prop}\label{crucial=0}
 Suppose $\kappa^s(M, \omega)\geq 0$, $F=\sqcup F_i$ is a maximal symplectic surface and each $F_i$ of positive genus.
If $(K_{\omega}+[F])^2=0$, then $\kappa^s(M, \omega)=0$ or $1$, and
each $F_i$ is a torus.

Moreover, suppose $\kappa^s(M, \omega)=0$ or $1$, $F=\sqcup F_i$ is
a maximal symplectic surface and each $F_i$ is a torus. If $(M,
\omega)$ is minimal, then $(K_{\omega}+[F])^2=0$ if and only if
$[F_i]^2=0$; and in general, suppose $(N, \sigma)$ is the minimal
model, then $(K_{\omega}+[F])^2=0$ if and only if  there is a
partition $ \{E_{i_j}\}$ of $\mathcal E_{M, \omega}$ such that
$[F_i]=\iota_*[F_i']-\sum_j E_{i_j}$, where $F_i'$ are disjoint
square 0 symplectic tori in $(N, \sigma)$.

 \end{prop}

When $\kappa^s(M, \omega)=-\infty$ we also have \eqref{KF<0} except
in one case.

 \begin{prop}\label{surface}
 Suppose $\kappa^s(M, \omega)=-\infty$, $F=\sqcup F_i$ is a maximal symplectic surface and each $F_i$ of positive genus.
  If  $F$ is  not a
section of a genus $g\geq 1$ $S^2$ bundle, then we have
\eqref{KF<0}.
 \end{prop}
\begin{proof}
This is also proved in \cite{LY} under the assumption that $F$ is
connected. The argument is a case by case analysis. Our argument
here is also a case by case analysis. As in \cite{LY}, we observe
that \eqref{KF<0} is equivalent to
\begin{equation}\label{misprint}
K_{\omega}^2-[F]^2\geq 4(1-g(F)).
\end{equation}
Here $g(F)$ is the  genus defined in Definition \ref{formal}. We
need to point out however that there is a misprint in (8) in
\cite{LY}: the lefthand side should be $K_{\omega}^2-[F]^2$.

With Lemma \ref{disc} understood we can check that the argument for
\eqref{misprint} in \cite{LY} for a connected $F$ remains valid in
each case for a disconnected $F$.
\end{proof}

We offer  another argument using SW thoery as in \cite{LiL}.

\begin{lemma}\label{gt}
Suppose $\kappa^s(M, \omega)=-\infty$, $e$ is a class with positive
$\omega-$genus $g_{\omega}(e)$ (see Definition \ref{formal}),
$e\cdot [\omega]>0$, and $e\cdot E> 0$ for any $E\in \mathcal E_{M,
\omega}$. If
 $e$ is not the class of a section of a genus $g\geq 1$ $S^2$ bundle, then
 $K_{\omega}+e$ is represented
by $\sqcup G_i$ with each $G_i$ a symplectic surface satisfying
\begin{equation} \label{u} [G_i]^2\geq 0, \quad -K_{\omega}\cdot
[G_i]+[G_i]^2\geq 0. \end{equation} In particular,
$(K_{\omega}+e)^2=\sum [G_i]^2\geq 0$.
\end{lemma}

\begin{proof} Recall that for a symplectic $4-$manifold $(M, \omega)$, there is a
canonical bijection between Spin$^c$ structures and $H_2(M;\mathbb
Z)$.  Recall also when $b^+(M)=1$ (which is our case here), for each
Spin$^c$ structure (equivalently, a class in $H_2(M;\mathbb Z)$),
there are two SW invariants, one of which is $SW_{\omega}$. By the
celebrated result of \cite{T}, if $SW_{\omega}(\alpha)\ne 0$, then
$\alpha\cdot [\omega]>0$. Moreover, if $\alpha\cdot E\geq 0$ for any
$E\in \mathcal E_{M, \omega}$, $\alpha$ is represented by a possibly
disconnected symplectic submanifold $\sqcup G_i$ satisfying
\eqref{u}.

We will  show that  $SW_{\omega}(K_{\omega}+e)\ne 0$ under the
assumption that $g_{\omega}(e)>0$ and $e\cdot [\omega]>0$. Since we
also assume $e\cdot E \geq 1$ for any $E\in \mathcal E_{M,
\omega}$, the conclusion of Lemma \ref{gt} will then follow.

 We first calculate  the Seiberg-Witten dimension of
the Spin$^c$ structure $K_{\omega}+e$,
$$\dim_{SW}(K_{\omega}+e)=-K_{\omega}\cdot
(K_{\omega}+e)+(K_{\omega}+e)^2=e(K_{\omega}+e)=2g_{\omega}(e)-2.$$
Since $g_{\omega}(e)$ is assumed to be positive,
$\dim_{SW}(K_{\omega}+e)\geq 0$. Notice that
$K_{\omega}-(K_{\omega}+e)=-e$. Thus we have
$$|SW_{\omega}(K_{\omega}+e)-SW_{\omega}(-e)|=\begin{cases}   1 & \hbox{if $(M, \omega)$ rational,}\cr
  |1-(e\cdot T)| &\hbox{if
  $(M, \omega)$ irrationally ruled,}\cr\end{cases}$$
  where $T$ is the unique positive fiber class of irrationally ruled manifolds (see \cite{McS}).
Since $(-e)\cdot [\omega]<0$ by assumption, we have
$SW_{\omega}(-e)=0$.

Hence we can conclude that unless $(M, \omega)$ is irrationally
ruled  and $e\cdot T=1$, we have $SW_{\omega}(K_{\omega}+e)\ne 0$.
It remains to show that $e\cdot T=1$ only if $(M, \omega)$ is an
$S^2$ bundle over a positive genus surface, i.e. $(M, \omega)$ is
minimal. This follows immediately from the fact that if $(M,
\omega)$ is not minimal, there are two classes $E_1, E_2\in \mathcal
E_{M, \omega}$ with $E_1+E_2=T$.
\end{proof}


 \begin{cor}\label{surface}
 Suppose $\kappa^s(M, \omega)=-\infty$, $F=\sqcup F_i$ is a maximal symplectic surface and each $F_i$ of positive genus.
Then  $(K_{\omega}+[F])^2=0$ if and only if $K_{\omega}+[F]$ is
represented by a disjoint union of symplectic spheres in the same
class with square $0$, or a disjoint union of symplectic tori
 whose classes have square $0$ and proportional to each other.
 \end{cor}

\begin{proof}
By Lemma \ref{gt} $K_{\omega}+[F]$ is represented by an embedded
symplectic surface $\sqcup G_i$ satisfying \eqref{u}.
Since $\sum [G_i]^2 =0$, we have $[G_i]^2=0$ for each $i$. Apply
\eqref{u} and the adjunction formula, we find that the genus of each
$G_i$ is either all equal to $0$ or $1$. Moreover, by the light cone
lemma, the classes $[G_i]$ must be proportional to each other.
\end{proof}

\subsubsection{$\kappa^s((K_{\omega}+[F])\cdot [\omega])$}
We now discuss the sign of $(K_{\omega}+[F])\cdot [\omega]$ for a
maximal symplectic surface $F$, in other words, we calculate
$\kappa^s((K_{\omega}+[F])\cdot [\omega])$.

\begin{prop}\label{omega}
If $\kappa^s(M, \omega)\geq 0$, then
\begin{equation} (K_{\omega}+[F])\cdot [\omega]\geq 0,
\end{equation}
with equality holds if and only if $(M, \omega)$ is minimal with
$\kappa^s=0$ and $F$ is empty.
\end{prop}
\begin{proof} When $\kappa^s(M, \omega)\geq 1$, we have
$K_{\omega}\cdot [\omega]>0$ and hence $(K_{\omega}+[F])\cdot
[\omega]>0$.

When $\kappa^s(M, \omega)=0$,  then $(K_{\omega}+[F])\cdot
[\omega]=0$ only when  $(M, \omega)$ is minimal and $[F]\cdot
[\omega]=0$. As $F$ is symplectic, this is possible only if $F$ is
empty.
\end{proof}

\begin{prop} \label{surface'} Suppose $\kappa^s(M, \omega)=-\infty$  and $F=\sqcup F_i\subset (M, \omega)$ is a possibly disconnected
 maximal symplectic surface with
 $g(F_i)\geq 1$. If $F$ is not a
section of a genus $g\geq 1$ $S^2$ bundle, then we have
\begin{equation}\label{adjoint4}
(K_{\omega}+[F])\cdot [\omega]\geq 0.
\end{equation}
Moreover, equality holds only if $[F]=-K_{\omega}$ and each $F_i$ is
a torus.
\end{prop}

\begin{proof} We first characterize those with $(K_{\omega}+[F])\cdot [\omega]= 0$.
If $F$ is not a section of a genus $g\geq 1$ $S^2$ bundle, then
$(K_{\omega}+[F])^2\geq 0$ by Proposition \ref{surface}. Notice that
$b^+(M)=1$. As $(K_{\omega}+[F])^2\geq 0$ and $[\omega]^2>0$, we can
apply the light cone lemma to $(K_{\omega}+[F])\cdot [\omega]=0$ to
conclude that $K_{\omega}+[F]$ is a torsion class. Since $M$ has no
torsion in homology, in fact, $[F]=-K_{\omega}$.

For any component $F_i$ of $-K_{\omega}$, we  have $-K_{\omega}\cdot
[F_i]=[F_i]^2$. Thus its genus is still $1$.

 It
remains to prove \eqref{adjoint4}.

By Proposition \ref{surface'} and Lemma \ref {disc} it suffices to
show that if $F$ is maximal and \begin{equation} \label{combine}
(K_{\omega}+[F])\cdot [\omega]<0 \quad\hbox{and} \quad
(K_{\omega}+[F])^2\geq 0,
\end{equation}then  we obtain a contradiction, often in the form $g(F)<1$, i.e. $2g(F)-2<0$. Our
argument is a case by case analysis.

\medskip
\noindent  $\bullet$  $S^2\times S^2$.

 In this case $K_{\omega}=-2H_1-2H_2$,
 $[F]=aH_1+bH_2$ for some integers $a$, $b$. Here $H_1$, $H_2$
  are classes of $S^2$ factors with positive symplectic area.
  Then
  $$\begin{array}{ll}
  K_{\omega}+[F]&=(a-2)H_1+(b-2)H_2,\\
  (K_{\omega}+[F])^2&=2(a-2)(b-2).
  \end{array}
  $$
 As $H_1$, $H_2$ have positive symplectic
area,
  $$\begin{array}{ll}
   [\omega]&=xH_1+yH_2, \quad y>0, x>0,\\
  (K_{\omega}+[F])\cdot [\omega]&=x(b-2)+y(a-2).
   \end{array}$$
 Then \eqref{combine} becomes that $$x(b-2)+y(a-2)<0, \quad (a-2)(b-2)\ge 0,$$
which implies that $a,b\leqslant 2$ and at most one of them gets the
value $2$.

If $[F]^2\ge 0$, since $H_1$ and $H_2$ are GW stable classes,  by
Lemma \ref{stable}, we know that $a,b\ge 0$. If $[F]^2<0$, then
$ab<0$, and so one of them should be $1$. It is straightforward to
check that in both cases, we have
$$2g(F)-2=(K_{\omega}+[F])\cdot [F]=(a-2)b+(b-2)a<0.$$


\medskip
\noindent $\bullet$  $\mathbb {CP}^2\sharp k\overline{\mathbb{CP}^2}$

Let $E_i$ be the positive generators of $H_2$ of the
$\overline{\mathbb{CP}^2}$ factors. In this case
$K_{\omega}=-3H+\sum_1^k E_i$ and $[F]=d[H]-\sum_1^k c_iE_i$ for
some $d>0$ and $c_i\ge 1$. Then $$\begin{array}{ll}
K_{\omega}+[F]&=(d-3)H-(c_i-1)E_i, \\
(K_{\omega}+[F])^2&=(d-3)^2-\sum (c_i-1)^2,\\

 [\omega]&= xH-\sum
z_iE_i, \quad x>0, z_i>0, x^2>\sum z_i^2,
\\
(K_{\omega}+[F])\cdot [\omega]&= (d-3)x-\sum z_i(c_i-1).
\end{array}$$
\eqref{combine} becomes
$$(d-3)x<\sum z_i(c_i-1), \quad (d-3)^2\ge \sum (1-c_i)^2.$$
Hence, when $d\ge 3$ we have the following absurd inequality

\begin{eqnarray*}
(d-3)^2x^2 <(\sum z_i(c_i-1))^2
 \le& \sum z_i^2 \cdot \sum (c_i-1)^2
 =(d-3)^2x^2.
\end{eqnarray*}
In fact, what is behind the inequality is the light cone lemma.
 Finally,  if $0<d<3$  then
$$2g(F)-2=(d-3)d-\sum(c_i-1)c_i<0.$$

\medskip
\noindent $\bullet$  Non-trivial $S^2-$bundle over $\Sigma_h$ with
$h\geq 1$.

 In this case let $U$ be the
class of a section with square $1$, $T$ be the class of  a fiber,
both  with positive symplectic area. Then $K_{\omega}=-2U+(2h-1)T$,
and $[F]=aU+bT$ for some integers $a$ and $b$.  Now $$
\begin{array}{ll} K_{\omega}+[F]&=(a-2)U+(2h-1+b)T,
\\
(K_{\omega}+[F])^2&= (a-2)(a-2+4h-2+2b).\end{array}
$$
 As  $U$, $T$ have positive
symplectic area,
    \begin{equation}\label{adbundle}\begin{array}{ll}[\omega]&=xU+yT, \quad x>0, x+y>0, x+2y>0,
    \\
(K_{\omega}+[F])\cdot [\omega]&=
(x+y)(a-2)+x(2h-1+b).\end{array}\end{equation}
 Then \eqref{combine} becomes
$$2h-1+b<-\dfrac{(x+y)(a-2)}{x}, \quad (a-2)(2h-1+b)\ge 0,$$
which implies that $a\le 2$ and $2h-1+b \le 0$, and at
most one equality holds. To proceed we compute that
\begin{equation}\label{s2bundle}
2g(F)-2=(K_{\omega}+[F])\cdot [F]=a(a-2+2h-1+b)+b(a-2).
\end{equation}

If $[F]^2\ge 0$, we also have $a\ge 0$ by Lemma \ref{stable} since
$T$ is a stable class. When $a=0$, then $b$ has to be positive as
$[F]\cdot [\omega]>0$, and by \eqref{s2bundle}, $2g(F)-2=-2b<0$; If
$a=2$ then $2h-1+b <0$ by \eqref{adbundle}, and by \eqref{s2bundle},
$2g(F)-2=2(b+2h-1)<0.$

If $[F]^2<0$, we have $a(a+2b)<0$.  When $a<0$, $a+2b>0$, then
$g(F)<0$ by \eqref{s2bundle}. The case when $a>0$ but $a\ne 1$ is
already analyzed above.

Finally, we analyze the case $a=1$. If $F$ is connected, then it is
a section. If it is not connected, then there is a component with
$a\ne 1$. But the genus of such a component (which is automatically
maximal as $M$ is minimal) is not positive as already shown.

\medskip
\noindent $\bullet$ $S^2\times \Sigma_h, h\geq 1$.

This case is similar to the previous case.

\medskip
\noindent $\bullet$ $(S^2\times \Sigma_h)\sharp k \overline{\mathbb
{CP}^2}$

Let $E_i$ be the positive generators of $H_2$ of the
$\overline{\mathbb{CP}^2}$ factors. In this case let $U$ be the
class of a section with square $0$, $T$ be the class of  a fiber,
both  with positive symplectic area. Then $\mathcal
E_{\omega}=\{E_i, T-E_i\}$ and
$$K_{\omega}=-2U+(2h-2)T+\sum_1^kE_i.$$ Thus $F$ is maximal if
and only if $$[F]=aU+bT-\sum_1^kc_iE_i, \quad a>c_i\geqslant 1.$$ We
explicitly compute,
$$\begin{array}{ll}K_{\omega}+[F]&=(a-2)U+(2h-2+b)T+\sum_1^k(1-c_i)E_i,
\quad a>c_i\geqslant 1,\\
(K_{\omega}+[F])^2&=2(a-2)(2h-2+b)-\sum (c_i-1)^2,\\

[\omega]&=xU+yT-\sum z_iE_i, \quad x, y, z_i>0, 2xy-\sum z_i^2>0,\\
(K_{\omega}+F)\cdot [\omega]&=(a-2)y+(2h-2+b)x-\sum z_i(c_i-1).
\end{array}$$
Then \eqref{combine} becomes  $$(a-2)y+(2h-2+b)x-\sum z_i(c_i-1)<0,
\quad \sum (c_i-1)^2\le 2(a-2)(b+2h-2).$$ When $(a-2)y+(2h-2+b)x\ge
0$,
\begin{eqnarray*}
((a-2)y+(2h-2+b)x)^2
 &<&(\sum z_i(c_i-1))^2\\
 &\le& \sum z_i^2 \cdot \sum (c_i-1)^2\\
 &<&2xy\cdot 2(a-2)(b+2h-2).
\end{eqnarray*}
This is equivalent to saying that $$((a-2)y+(2h-2+b)x)^2<0,$$ which
is a contradiction! Again what is hidden behind is the light cone
lemma.

Now let us suppose $(a-2)y+(2h-2+b)x<0$. If $a<2$,  then by the
maximality condition $a>c_i\ge 1$,  $(M, \omega)$ is in fact minimal
in which case we have treated above. Now, we assume that $a\ge 2$
and $2h-2+b<0$. Then
$$\sum (c_i-1)^2\le 2(a-2)(b+2h-2)\le 0.$$ This forces
$c_i=1$ and $a=2$. In this case $2g(F)-2=2(2h-2+b)<0$.
\end{proof}


\subsection{Existence and Uniqueness of relatively minimal model}
Any  surface can be   made  maximal by blowing down.

\begin{lemma} \label{become-maximal}Suppose $F\subset (M, \omega)$ is a symplectic surface
without sphere components.
 Denote the set of $E$ with $ F\cdot E=0$ by
$\mathcal E_{\omega}^{F}$. Suppose $\{E_i\}\subset \mathcal
E_{\omega}$ is a maximal subset of pairwise orthogonal elements.
Blow down a set of  symplectic $-1$ spheres $S_i$ in the classes
$\{E_i\}$, which are   disjoint from each other and  from $F$,  to
obtain $(M', \omega')$. If we denote the same surface in
$(M,',\omega')$ by $F'$, then  $F'$ is maximal in $(M', \omega')$.
\end{lemma}
\begin{proof} When $F$ is connected, this is Theorem 1.1(ii) in \cite{Mcjams}. For a disconnected $F$
it  follows from Theorem 3.4 in \cite{Mcjams}, with $\Lambda$ there
being the subgroup orthogonal to the subgroup generated by $[F_i]$.
\end{proof}

\begin{definition} Suppose $F\subset (M, \omega)$ is a symplectic surface
without sphere components. $(M',  \omega', F')$  in Lemma
\ref{become-maximal} is called a relative minimal model of $(M,
\omega, F)$.

For a general symplectic surface $F$, a relative minimal model of
$(M, \omega, F)$ is defined to be a relative minimal model of $(M,
\omega, F^+)$.
\end{definition}

It is well known that in the case of $\kappa^s=-\infty$, there are
more than one minimal models. So the following  uniqueness of
relative minimal model when $F^+$ is not empty is surprising.

\begin{theorem}\label{unique}
If $F^+$ is nonempty, there is a unique relative minimal model.
\end{theorem}


\begin{proof} Without loss of generality we can assume  $F=F^+$.

Recall  that $$\mathcal E_{\omega}^{F}=\{E\in \mathcal E_{\omega}|
E\cdot [F]=0\}.$$

When $M$ is not rational or ruled, the classes in  $\mathcal
E_{\omega}$ are pairwisely orthogonal, and represented by disjoint
symplectic $-1$ spheres. Of course the same is true for  $\mathcal
E_{\omega}^{ F}$. Thus there is a unique way to make $F$ maximal.

When $M$ is irrationally ruled, $\mathcal E_{\omega}$ can be
described as
$$\{E_1, T-E_1, \cdots, E_l, T-E_l\},$$
where $T$ is the unique $\omega-$positive fiber class.
 If $\mathcal E_{\omega}^{F}$ contains both $E_1$ and
$T-E_1$, then $[F]\cdot T=0$. As $T$ is a GW stable class, we have
$F_i\cdot T=0$ as well by Lemma \ref{stable}. It follows that each
$[F_i]$ is of the form $a_i T -\sum c_j E_j$, $a_i\geq 0$. But by
the adjunction formula, such a component has genus at most zero. As
each component of $F$ is of positive genus, $\mathcal
E_{\omega}^{F}$ contains only pairwisely orthogonal classes. Again
there is a unique way to make $F$ maximal in this case.

 If  $M=\mathbb {CP}^2\# \overline{\mathbb {CP}^2}$, there is a
unique class in $\mathcal E_{\omega}$ and hence at most one class in
$\mathcal E_{\omega}^{F}$.

The remaining case is $M=\mathbb {CP}^2\# l\overline{\mathbb
{CP}^2}$ with $l\geq 2$. The proof is based on the properties of the
adjoint class of a {\it maximal} surface established in 3.2.

 Suppose
$(M',\omega', F')$ is a relative minimal model of $(M, \omega, F)$.
Then by Lemmas \ref{surface} and \ref{surface'} we can assume that
\begin{equation}\label{'}\begin{array}{ll}
(K_{\omega'}+[F'])^2&\geq 0,\cr
(K_{\omega'}+[F'])\cdot[\omega']&\geq 0,
\end{array}
\end{equation}
 since $M$ is not an $S^2-$bundle over
$\Sigma_h$ with $h\geq 1$. Let $S_i$ be a set of disjoint symplectic
$-1$ spheres which are blown down to obtain $(M', \omega')$. Notice
that the  $S_i$ are also assumed to be disjoint form $F$. Let
$\mathcal U=\{E_i=[S_i]\}$. Then
\begin{equation}\label{canonical classes}
K_{\omega}=\pi^* K_{\omega'}+\sum E_i, \end{equation}

 Suppose $G\in \mathcal
E_{\omega}^{F}$ and is distinct from $E_i$. Suppose also that  there
is some $E_i\in \mathcal U$ such that  $E_i\cdot G>0$. After
choosing a symplectic $-1$ sphere in the class $G$ which is disjoint
from $F$ and intersects the $S_i$ transversally and non-negatively,
by Lemma \ref{blowdown}  we see that there is possibly immersed
symplectic surface $C$ in $(M', \omega')$ with the following
properties:

$\bullet$ $C$ is disjoint from $F'=F$, so
\begin{equation}\label{FC} [ F']\cdot [C]=0.
\end{equation}.

$\bullet$  $[C]$ is related to $G$ via
\begin{equation}\label{CG}
G=\iota_*[C]-\sum (E_i\cdot G)E_i.
\end{equation}

$\bullet$  By \eqref{CG}, we have
\begin{equation}\begin{array}{ll}
[C]^2&=[G]^2+2\sum (E_i\cdot G)^2-\sum(E_i\cdot
G)^2\\
&=-1+\sum(E_i\cdot G)^2\\
&\geq 0.\end{array}
\end{equation}

$\bullet$   By \eqref{canonical classes} and \eqref{CG}, we have
\begin{equation}\begin{array}{ll}\label{KC} K_{\omega'}\cdot [C]&=\pi^*K_{\omega'}\cdot \iota_*[C]\\
&=(K_{\omega}-\sum E_i)(G+\sum(E_i\cdot G)E_i)\\
&=K_{\omega}\cdot G +\sum (E_i\cdot G)(K_{\omega}\cdot E_i-E_i^2-1)\\
&= K_{\omega}\cdot G-\sum (E_i\cdot G)\\
&= -1- \sum E_i\cdot G\\
&<0
\end{array}
\end{equation}

$\bullet$ By \eqref{FC} and \eqref{KC} we conclude
\begin{equation}\label{hard} (K_{ \omega'}+[F'])\cdot [C]<0.
\end{equation}

Notice that both $[C]$, $K_{\omega'}+[F']$ have non-negative
square by \eqref{CG} and \eqref{'}, and both pair positively with
$[\omega']$ by \eqref{'}. Since $b^+(M)=1$, \eqref{hard} violates
the light cone lemma. This contradiction again shows that $\mathcal
E_{\omega}^{F}$ contains only pairwisely orthogonal classes.
Therefore there is a unique way to make $F$ maximal in this case.
\end{proof}

We remark that there is an alternative argument when  $b^+(M)=1$ and
there is a component, say $F_1$, with $[F_1]^2\geq 0$. In this case
we can directly show that the classes in $\mathcal
E_{\omega}^{\tilde F}$ are pairwise orthogonal.
 Suppose $G_1, G_2 \in \mathcal
E_{\omega}^{\tilde F}$ and $G_1\cdot G_2\ne 0$. Then $G_1\cdot
G_2>0$. If $G_1\cdot G_2\geq 2$, then $(G_1+G_2)^2>0$.  Since
$b^+(M)=1$, this contradicts to the light cone lemma as $[F_1]^2\geq
0$ and $[F_1]\cdot (G_1+G_2)=0$. If $G_1\cdot G_2=1$ then
$(G_1+G_2)^2=0$, we still get a contradiction unless $[F_1]$ and
$G_1+G_2$ are proportional to each other. However, this is
impossible due to the adjunction formula and $K_{\omega}\cdot
(G_1+G_2)=-2$.

When $[F']=-K_{\omega'}$ we can also directly argue that if  $G$ is
a $-1$ class of $(M, \omega)$ distinct from $E_i$, then $G$ does not
lie in $\mathcal E_{\omega}^{F}$.
 Notice that $$G\cdot E_i\geq 0\quad \hbox{and} \quad
K_{\omega}\cdot G=-1,$$
 and hence by \eqref{canonical classes}
$$-[F]\cdot G=\pi^*K_{\omega'}\cdot G=K_{\omega}\cdot
G-\sum (E_i\cdot G)\leq -1.$$

\subsection{$\kappa^s(M, \omega, F)$}
In this subsection we define the relative Kodaira dimension of a
$4-$dimensional symplectic manifold relative to a possibly
disconnected, embedded  symplectic surface.

\subsubsection{Definition for a
 maximal $F$ without sphere components} We first assume that $F$ is
 maximal and has no sphere components.
 {\definition \label{symp Kod}
Let $F\subset(M, \omega)$ be a  maximal symplectic surface without
sphere components.
  Then the relative Kodaira dimension of
$(M,F,\omega)$
 is defined in the following way:
if $F$ is empty, then $(M,\omega)$ is necessarily minimal and
$\kappa^s(M, \omega, F)$ is defined to be $\kappa^s(M, \omega)$.
Otherwise,

\[
\kappa^s(M, \omega, F)=\left\{\begin{array}{cc}
-\infty & \hbox{if $(K_{\omega}+[{F}])\cdot \omega<0$ or\,\, $(K_{\omega}+[{F}])^2<0$},\\
0& \hbox{ if $(K_{\omega}+[{F}])\cdot \omega=0$ and $(K_{\omega}+[{F}])^2=0$},\\
1& \hbox{ if $(K_{\omega}+[{F}])\cdot \omega>0$ and $(K_{\omega}+[{F}])^2=0$},\\
2& \hbox{ if $(K_{\omega}+[{F}])\cdot \omega>0$ and $(K_{\omega}+[{F}])^2>0$}.\\
\end{array}\right.
\]}


Next, we prove that the above definition is well defined.

\begin{theorem} \label{well def}
Definition \ref{symp Kod} is well-defined.
\end{theorem}
\begin{proof}
The only thing we need to check is that there  is no maximal surface
without sphere components  $F\subset (M, \omega)$ with
$$(K_{\omega}+[F])\cdot [\omega]=0, \quad \hbox{and}\quad
(K_{\omega}+[F])^2>0.$$ By Proposition \ref{omega} it remains to
discuss the case when $M$ is rational or ruled. As $b^+(M)=1$ in
this case, the statement follows from the light cone lemma and
$[\omega]^2\geq 0$.

\end{proof}

As mentioned in the introduction, the main result in
\cite{LY} has the following simple interpretation.

\begin{theorem} \label{fibersum}
Let $(M,\omega)$ be a $4-$dimensional relatively minimal fiber sum
of $(M_1,\omega_1)$ and $(M_2,\omega_2)$ along connected genus
$g\geq 1$ symplectic surfaces $F_i\subset (M_i,\omega_i)$. Then
\begin{equation}\label{sum4}\kappa^s(M, \omega)=\max\{\kappa^s(M_1,\omega_1,
F_1),\kappa^s(M_2,\omega_2, F_2)\}.\end{equation}
\end{theorem}





\subsubsection{Comparing the relative and absolute Kodaira dimensions}
\begin{theorem} \label{preclass}
Assume $F$ is a maximal symplectic surface without sphere components
in $(M, \omega)$, then
\begin{equation}\label{compare}\kappa^s(M, \omega, F)\ge \kappa^s(M, \omega).
\end{equation}

\end{theorem}
\begin{proof}
 \eqref{compare}
certainly holds when $\kappa^s(M,\omega, F)=2$.

To deal with the case of  $\kappa^s(M,\omega, F)=1$, let us
introduce $(N, \sigma), K_{\sigma}, c_i, F'$ as in the proof of
Proposition \ref{crucial}. We can assume that $\kappa^s(M,
\omega)\geq 0$, otherwise the inequality \eqref{compare} holds
automatically. Recall that  it is shown in the proof of Proposition
\ref{crucial} that $K_{\sigma}\cdot [F']\geq 0$. As
$(K_{\omega}+[F])^2=0$, it follows that
$$\begin{array}{ll}
K_{\sigma}^2&=(K_{\sigma}+[F'])^2-(2K_{\sigma}\cdot [F']+[F']^2)\cr
&=(K_{\omega}+[F])^2+\sum(c_i-1)^2-(K_{\sigma}\cdot
[F']+[F']^2)-K_{\sigma}\cdot [F']\cr &=
\sum(c_i-1)^2-(K_{\omega}\cdot[F]+[F]^2)-\sum
(c_i^2-c_i)-K_{\sigma}\cdot [F']\cr &\leq \sum(1-c_i) \leq 0.
\end{array}
$$
Thus, in this case, we also have $\kappa^s(M, \omega)=\kappa^s(N,
\sigma)\le 1=\kappa^s(M,\omega, F)$.

If $\kappa^s(M, \omega, F)=0$, then $(K_{\omega}+[F])\cdot
[\omega]=0$. By Proposition \ref{omega}, $\kappa^s(M,
\omega)=-\infty$ when $F\neq \emptyset$, and $\kappa^s(M, \omega)=0$
when $F= \emptyset$.

Now let us check the case of $\kappa^s(M, \omega, F)=-\infty$. If
$(K_{\omega}+[F])\cdot \omega<0$, we have $\kappa^s(M,
\omega)=-\infty$
  by Proposition \ref{omega}.
If $(K_{\omega}+[F])^2<0$, then $\kappa^s(M, \omega)=-\infty$ by
Propositions \ref{crucial} and \ref{surface'}.

\end{proof}


\subsubsection{Classification when $\kappa^s(M, \omega, F)=-\infty$}

\begin{theorem} \label{-infty} Suppose a nonempty surface $F\subset(M, \omega)$ is maximal with each component positive genus. Then
 $\kappa^s(M, \omega, F)=-\infty$  if and only if $M$ is  a genus $h$ $S^2$ bundle with $h\ge 1$,
 and $F$ is a
section.


\end{theorem}

\begin{proof}
By Theorem \ref{preclass} $M$ satisfies $\kappa^s(M)=-\infty$. Thus
the only if part of the statement is a direct consequence of
Propositions \ref{surface} and \ref{surface'}.

Let us verify the if part. Suppose either  $M$ is $S^2 \times
\Sigma_h$, $[F]=[\Sigma_h]+b[S^2]$, or
 $M$ is a nontrivial $S^2$ bundle over $\Sigma_h$, $[F]=U+bT$.

We cheek the case of  $S^2\times \Sigma_h$, the other case is
similar. As in 3.2, we compute in this case
$$\begin{array}{ll}(K_{\omega}+[F])^2&=-(b+2h-2),\\
(K_{\omega}+[F])\cdot [\omega]&=-y+(b+2h-2)x.
\end{array}
$$
If $(K_{\omega}+[F])^2<0$, then $\kappa^s(M,  \omega, F)=-\infty$.
If $(K_{\omega}+[F])^2\geq 0$, then $b+2h-2\leq 0$. Since $x>0,
y>0$, we have $(K_{\omega}+[F])\cdot [\omega]<0$, so $\kappa^s(M,
\omega, F)=-\infty$ as well.

\end{proof}


\begin{remark} \label{-infty'} Notice that this classification in Theorems \ref{-infty}  is
independent of $\omega$. This may not be so obvious, and  actually
it follows from Theorem \ref{fibersum} as summing with an
$S^2-$bundle along a section is the so called smoothly trivial sum.
\end{remark}


\subsubsection{Classification when $\kappa^s(M, \omega, F)=0$}
 By Theorem \ref{preclass} and  Proposition
\ref{surface'}, we have

\noindent \begin{theorem} \label{classification} Suppose a nonempty
surface $F\subset(M, \omega)$ is maximal with each component
positive genus. $\kappa^s(M, \omega, F)=0$ if and only if
$$\kappa^s(M, \omega)=-\infty \quad \hbox{and}\quad  [F]=-K_\omega.$$

\end{theorem}
\subsubsection{Dependence on $F$}

\begin{prop}\label{independence}
Suppose $ F_1, F_2\subset (M, \omega)$ are maximal symplectic
surfaces without sphere components. If $[F_1]=[F_2]$, then
$\kappa^s(M, \omega, F_1)=\kappa^s(M, \omega, F_2)$.
\end{prop}
\begin{proof}
By the classification Theorems \ref{-infty} and
\ref{classification}, we can assume that $\kappa^s(M, \omega,
F_i)\geq 1$ for $i=1,2$. Suppose $\kappa^s(M,  \omega, F_1)=1$, then
$$(K_{\omega}+[F_2])^2=(K_{\omega}+[F_1])^2=0,$$ so
$\kappa^s(M,  \omega, F_2)$ is at most $1$. Thus $\kappa^s(M,
\omega, F_2)$ must be equal to $1$ as well.
\end{proof}


\subsubsection{Non-maximal surface} We have defined $\kappa^s(M,
\omega, F)$ when $F$ is empty, or
 maximal and without sphere components.
As a direct consequence of  Lemma \ref{max} and Theorem
\ref{unique}, we can extend $\kappa^s(M, \omega, F)$ to any embedded
symplectic surface $F$ without sphere components.

\begin{definition}\label{nonminimal} Suppose $F\subset(M, \omega)$
is a symplectic surface without sphere components. If $F=\emptyset$,
then   the relative Kodaira dimension of $(M,\omega, F)$,
$\kappa^s(M, \omega, F)$, is defined to be $\kappa^s(M, \omega)$.
Otherwise, let $(M',\omega', F')$ be the unique relative minimal
model of $(M, \omega, F)$, and define $\kappa^s(M, \omega, F)$ to be
$\kappa^s(M',\omega',  F')$.
\end{definition}

It is easy to see that all the results for maximal surfaces hold for
general surfaces with obvious modifications.

\subsubsection{$F$ possibly with sphere components}

Recall that  $ F^+$ is the surface obtained from $F$ by removing the
sphere components.

 {\definition \label{symp Kod}
Let $F\subset(M, \omega)$ be an embedded symplectic surface. Then
the relative Kodaira dimension of $(M,\omega, F)$, $\kappa^s(M,
\omega, F)$,  is defined to be $\kappa^s(M, \omega, F^+)$.}


It is not hard to check that it is still well-defined and all the
results  still hold in this more general setting with obvious
modifications.


 We notice that the above definition is similar in one aspect to the
definition of the Thurston norm of  $3-$manifolds: the $2-$spheres
have to be discarded. One explanation is that a $2-$sphere has
$\kappa^t=-\infty$, so it behaves like the empty set in some sense.

It is also necessary in our case for two reasons, one is the
positive genus assumption in several results in section 3, e.g.
Lemma \ref{max}. Another is that there are the following three
special situatons with $F$ a sphere,  which would have relative
dimension $-\infty$ if we had defined it ``naively'':
\begin{enumerate}
\item  $K_{\omega}^2=0$, $K_{\omega}\cdot [F]=0$, $[F]^2=-2$.

\item  $K_{\omega}^2=0$, $K_{\omega}\cdot [F]=1$, $[F]^2=-3$.

\item  $K_{\omega}^2=1$, $K_{\omega}\cdot [F]=0$, $[F]^2=-2$.

\end{enumerate}
 An example for (1) is  $M=E(2)$ and $F$ a $-2$ sphere, and
an example for (2) is  $M=E(3)$ and $F$ a $-3$ sphere.

Due to Proposition \ref{independence}, it is  possible to extend
$\kappa^s(M, \omega, F)$ to the case of $F$ being a symplectic
surface with pseudo-holomorphic singularities, or a weighted
symplectic surface. We should also mention that the notion of the
logarithmic Kodaira dimension of a noncomplete variety introduced by
Iitaka (see \cite{Ib}) should be closely related to our relative
Kodaira dimension $\kappa^s(M, \omega, F)$. All these will be
studied elsewhere.

\section{Relative Kod. dim. in dim. $2$ and  fibrations over a surface} \label{rel} In this section we
introduce  Kodaira dimension for a $2-$manifold relative to a
rational linear combination of points, and discuss how it might be
used  to compute the Kodaira dimension of the total space of certain
fibrations with a $2-$dimensional base or a $2-$dimensional fiber.

In general, our viewpoint for a fibration is: ``good'' fibers and a
``singular'' base. More precisely, we first project the singular
fibers to the base to obtain a finite set. We then assign a rational
weight for each point of the image, subject to the requirement that
the weight is positive and only depends on the type (local data) of
the singular fiber. For any such assignment,
 we get an effective $\mathbb Q-$divisor on the base, hence
relative Kodaira dimension for the base along with absolute Kodaira
dimensions for the fiber and the total space. What we are able to
show is that often there is a  way (and sometimes unique)  to assign
the weight  so that
 these three
quantities together  form an additivity relation. We also note that
this scheme does not work in all cases. For example, we observe that
for a  genus two $4-$dimensional Lefschetz fibration over $S^2$
with non-minimal total space, we have to further modify this scheme
taking into account the total intersection numbers  of $-1$ classes
with the fiber class, in particular, we also need to relativize the
Kodaira dimension of the fiber. It indicates that relative Kodaira
dimension might be related to divisor contractions.



\subsection{$\kappa^t(F, D)$,  Riemann-Hurwitz formula and Seifert fibrations}

In dimension $2$, codimension $2$ submanifolds are just points.
\begin{definition}
Let $F$ be a closed oriented real surface. A $\mathbb Q$ linear
combination of points on $F$ of the form $D=\sum_{i=1}^k m_i x_i,
x_i\in F, m_i\in \mathbb Q$ is called a $\mathbb Q-$divisor on $F$.
Denote by $c(D)=\sum_{i=1}^k m_i$. The set $\{x_i\}$ is called the
support of $D$. $D$ is called effective or positive if $m_i\geq 0$,
and $D$ is called negative if $m_i\leq 0$.
\end{definition}

\begin{definition}\label{rel 2dim}
Let $F$ be a closed oriented real surface of genus $g$ and $D$ a
$\mathbb Q-$divisor. Define
\[
\kappa^t(F,D)=\left\{\begin{array}{cc}
-\infty & \hbox{ if $2g-2+c(D)<0$},\\
0& \hbox{ if $2g-2+c(D)=0$},\\
1& \hbox{ if $2g-2+c(D)<0$}.
\end{array}\right.
\]

\end{definition}

$D$ is allowed to be the empty set, and in this case, $\kappa^t(F,
\emptyset)=\kappa^t(F)$.
 Clearly,  $\kappa^t(F, D)\geq \kappa^t(F)$ if $D$
is effective, and $\kappa^t(F, D)\leq \kappa^t(F)$ if $D$ is
negative.

For an integral and effective $D$,  there are simple analogues of
$4-$dimensional results. For instance, if a nonempty $D$ is integral
and effective, then $\kappa^t(F, D)=-\infty$ if and only if $F=S^2$
and $D=x$ for some $x\in F$. Notice that if we view $S^2$ as an
$S^2-$bundle over a point, then this simple fact exactly corresponds
to Theorem \ref{-infty}.

We also observe that the relative Kodaira dimension fits well with
the connected sum construction (compare with Theorem
\ref{fibersum}).

\begin{prop}Suppose $F$ is the connected sum of $F_1$ and $ F_2$ along $p_1,...,
p_n\in F_1, q_1,..., q_n\in F_2$, then
$\kappa^t(F)=\max\{\kappa^t(F_1,D_1),\kappa^t(F_2,D_2)\}$ with
$D_1=\sum_{i=1}^n p_i, D_2=\sum_{i=1}^n q_i$. \end{prop}

 If  $n=0$, then $D_1=D_2=\emptyset$. By the definition of
$\kappa^t$ for a disconnected manifold,
$\kappa^t(F)=\kappa^t(F_1\sqcup F_2)=\max\{\kappa^t(F_1),
\kappa^t(F_2)\}$. When $n=c(D_1)=c(D_2)$ is positive, this can also
be easily checked.



As mentioned in the introduction, $\kappa^t(F, D)$ is introduced to
achieve additivity of Kodaira dimensions for a fibration where $F$
is either the base or a smooth fiber. When  $F$ is the base, the
support of $D$ is often the image of the singular fibers, and each
weight $m_i$ is positive. It might be delicate to determine
 the exact value of $m_i$ in each specific case.
We will illustrate this idea by investigating several types of
important fibrations. We begin with ramified coverings in dimension
$2$.

\subsubsection{Ramified coverings  and
the Riemann-Hurwitz formula} \label{2D} Let $S', S$ be oriented
surfaces and $\pi: S' \longrightarrow S$ a ramified cover of degree
$N$. Suppose the ramification set is $\{p_i\}$ and denote by
$e_{p_i}$ the ramification index of $p_i$. Then we have the famous
Riemann-Hurwitz formula:
\begin{equation}\label{hurwitz}\chi(S')=
N\chi(S)-\sum(e_{p_i}-1)=N(\chi(S)-\frac{1}{N}\sum(e_{p_i}-1)).\end{equation}


A ramified cover is often viewed as a fibration with ``good'' base
and some ``bad'' fibers. However,  we would like to think of  the
base surface $S$ as a ``relative surface" $(S,D)$ with
$$D_{\pi}=\sum_{\{p_i\}}\dfrac{e_{p_i}-1}{N}\, p_i.$$
With this natural choice of $D_{\pi}$, the Riemann-Hurwitz formula
\eqref{hurwitz} can be interpreted as
$$\kappa^t(S')=\kappa^t(S, D_{\pi})+\kappa^t(fiber)=\kappa^t(S,
D_{\pi}).$$



\subsubsection{Seifert fibrations} \label{seifert}

A Seifert fibration on a $3-$manifold $M^3$ is a fibration
$\pi:M^3\to B$ to a closed surface $B$ with circle fibers.
The singular fibers are all multiple fibers. Suppose the singular
fibers have images  $p_1, ..., p_n\in B$ and multiplicities  $a_1,
..., a_n$.
Classically, $B$ is viewed as an orbifold with orbifold points
$\{p_i\}$, and with  orbifold Euler characteristic
$$\chi^{orb}(B)=\chi(B)-\sum(1-\frac{1}{a_i}).$$
Our view is slightly different, viewing the base as a relative
surface with the natural choice of divisor, $D_{\pi}=
\sum_{i=1}^n(1-\frac{1}{a_i})p_i$, suggested by the definition of
$\chi^{orb}(B)$ above.

\begin{prop} \label{12} With the set up above, we have
$$\kappa^t(M^3)=\kappa^t(B, D_{\pi})+\kappa^t(fiber)=\kappa^t(B,
D_{\pi}). $$

\end{prop}
The argument is  similar to the special case of  $S^1-$bundles in
\cite{Z}. Notice  that $\kappa^t(B, D_{\pi})$ only depends on the
sign of $\chi^{orb}(B)$.

When $\chi^{orb}(B)>0$, by the classification of Seifert fibre
spaces, $M^3$ has $S^3$  geometry if $\pi_1(M^3)$  is finite, and
$S^2 \times \mathbb{R}$  geometry if $\pi_1(M^3)$ is infinite. In
this case, $\kappa^t(M^3)=\kappa^t(B, D_{\pi})=-\infty$.

When $\chi^{orb}(B)=0$, again by the classification, the possible
geometries for $M^3$ are Euclidean or Nil. In this case,
$\kappa^t(M^3)=\kappa^t(B, D_{\pi})=0$.

Finally, when $\chi^{orb}(B)<0$, $M^3$ has geometry of type
$\mathbb{H}^2\times \mathbb{R}$ or $\widetilde{SL_2(\mathbb{R})}$.
In this case, $\kappa^t(M^3)=\kappa^t(B, D_{\pi})=1$.
\subsection{Lefschetz fibrations}

Now we investigate several kinds of $4-$dimensional Lefschetz
fibrations. We will denote a $4-$dimensional Lefschetz fibration by
$\pi: M^4\to B$, and a general smooth fiber by $F$. It suffices to
restrict to relatively minimal Lefschetz fibrations. By Proposition
\ref{bundle}, we can also assume that there is at least a singular
fiber.
\subsubsection{When $\kappa^t(F)=-\infty$}
 Notice that if  $\kappa^t(F)=-\infty$, if
there is a singular fiber, then it is not relatively minimal. So we
also assume from now on that $\kappa^t(F)\geq 0$.

\subsubsection{When $\kappa^t(B)=1$ and $\kappa^t(F)\geq 0$}
In this case it was shown in \cite{DZ} that if $M^4$ admits a
symplectic (complex) structure, then
\begin{equation}\label{sh}\kappa^{s(h)}(M^4)=\kappa^t(F)+\kappa^t(B).
\end{equation}
Since for any effective divisor $D$ of the base surface $B$, we have
$\kappa^t(B, D)=\kappa^t(B)=1$,  if we assign any positive weight
$b_i$, we still have
\begin{equation}\label{shr}\kappa^{s(h)}(M^4)=\kappa^t(F)+\kappa^t(B, D_{\pi, b_i}).
\end{equation}
\subsubsection{When $\kappa^t(B)=0$ and $\kappa^t(F)=0$}
In this case, it was calculated in \cite{DZ} that $\kappa^s(M^4)=1$.
Thus for any positive assignment $b_i$, we have
$$\kappa^s(M^4)=1=1+0=\kappa^t(B, D_{\pi, b_i})+\kappa^t(F).$$

\subsubsection{When $\kappa^t(B)=-\infty$ and $\kappa^t(F)=0$}\label{4D}
In this  case there is a unique choice of weights. Notice that there
is only one type of elliptic Lefschetz singular fibers. Thus the
weight $b$ is determined by a fibration $\pi:K3\to S^2$ with $24$
singular fibers:  If the additivity holds for this fibration, then
$\kappa^t(S^2, D_{\pi, b})=\kappa^s(K3)-\kappa^t(F)=0-0=0$, which
means that $-\chi(S^2)-c(D_{\pi, b})=2-24b=0$, i.e.
$b=\frac{1}{12}$. Then it is easy to check that with this choice of
weight, the additivity also holds  for all relatively minimal
elliptic Lefschetz fibrations over $S^2$, namely, $\pi:E(n)\to S^2$
with $12n$ singular fibers, as $\kappa^s(E(n))=\kappa^s(n-2)$.

In the remaining cases we assume the fibration is hyperelliptic.
\subsubsection{When $\kappa^t(B)=0$ and $\kappa^t(F)=1$ and the
fibration is hyperelliptic} In this case, it was calculated in
\cite{DZ} that $\kappa^s(M^4)=2$. Thus for any positive assignment
$b_i$, we have
$$\kappa^s(M^4)=1=1+0=\kappa^t(B, D_{\pi, b_i})+\kappa^t(F).$$
\subsubsection{When $\kappa^t(B)=-\infty, \kappa^t(F)=1$ and the
fibration is hyperelliptic with minimal total space} In this case,
since $F$ is not a torus,  the total space $M^4$ admits a compatible
symplectic structure $\omega$. We further observe
\begin{lemma} For any genus $g\geq 2$ Lefschetz fibration with
minimal total space $M$ and a compatible symplectic form $\omega$,
\begin{equation}\label{k2'} \kappa^s(M, \omega)=\kappa^s(K_{\omega}^2)+1=\kappa^s(K_{\omega}^2)+\kappa^t(F).
\end{equation}
\end{lemma}
\begin{proof} \eqref{k2'} certain holds if $K_{\omega}^2<0$.
If  $K_{\omega}^2=0$, then $\kappa^s(M)=0$ or $1$. Since $g\geq 2$,
by the adjunction formula, we actually must have
$\kappa^s(M)=1=\kappa^s(0)+1$. The remaining case is
$K_{\omega}^2>0$.  If $\kappa^s(M, \omega)=-\infty$ and
$K_{\omega}^2>0$, then $M$ is $\mathbb {CP}^2$ or an  $S^2-$bundle
over $S^2$ or $T^2$. It is easy to check that there are no square
$0$ symplectic surfaces with genus at least $2$ in such a $(M,
\omega)$. Thus, we must have in this case
$\kappa^s(M)=2=1+1=\kappa^s(K_{\omega}^2)+1$.
\end{proof}

It remains to show that we can assign $b_i$ to each $x_i$ such that
$b_i$ only depends on the singularity type of $\pi^{-1}(x_i)$ and
\begin{equation}\label{=}\kappa^s(S^2, D_{\pi,
b_i})=\kappa^s(K_{\omega}^2).\end{equation}

 We take clues from Endo's signature
formula for hyperelliptic fibration over $S^2$ (\cite{En}):
\begin{equation}
\sigma(M)=-\frac{g+1}{2g+1}a+\sum_{p=1}^{[\frac
{g}{2}]}(\frac{4p(g-p)}{2g+1}-1)s_p,
\end{equation}
where $a$ is the number of non-separating singular fibers, and $s_p$
is the number of separating fibers of type $(p, g-p)$. The formula
for $K_{\omega}^2$ is calculated in \cite{DZ} to be
\begin{equation}\label{k2}\begin{array}{ll}K_{\omega}^2&=3\sigma(M)+2\chi(M)\\
&=2(2(2-2g))+\frac{g-1}{2g+1}a+\sum_{p=1}^{[\frac
{g}{2}]}\frac{6p(g-2p)+2g(p-1)+(4gp-1)}{(2g+1)}s_p.
\end{array}
\end{equation}

Let $b_{g, ns}$ be the weight for a non-separating fiber and $b_{g,
p}$ be the weight for a separating fiber of type $(p, g-p)$. By
\eqref{k2} it is natural to propose that
   \begin{equation}\label{bns}\begin{array}{ll}b_{g,ns}&=\frac{g-1}{(4g-4)(2g+1)}=\frac{1}{4(2g+1)},\cr
   b_{g, p}&= \frac{6p(g-2p)+2g(p-1)+(4gp-1)}{(4g-4)(2g+1)},\end{array}\end{equation}
and it is easy to check that, with this choice of $b_i$, \eqref{=}
holds.

In fact, $b_i$ defined by \eqref{bns} should be the unique weight
such that \eqref{=} holds. We have indeed verified the uniqueness
for genus 2 fibrations. In this case, in \eqref{bns},  $b_{2,
ns}=\frac{1}{20}$, and
 $b_{2,1}=\frac{7}{20}$. Our strategy is simple.  First consider a self  fiber
 sum of
a genus two holomorphic Lefschetz fibration with no separating
 fibers and $20$ non-separating singular fiber in (\cite{C}). It is minimal by \cite{U1} and has
 $K_{\omega}^2=0$. Thus it follows $b_{2, ns}$ has to be
 $\frac{1}{20}$. We next consider a self  fiber
 sum of
a genus two  Lefschetz fibration with $2$ separating fiber and $6$
non-separating singular fiber in \cite{Ma} and \cite{OS}. It is minimal again by
\cite{U1} and also has
 $K_{\omega}^2=0$. With $b_{2, ns}$ already determined to be $\frac{1}{20}$, $b_{2, 1}$ has to
 be $\frac{7}{20}$.

\subsubsection{When $\kappa^t=-\infty, \kappa^t(F)=1$ and the
fibration is hyperelliptic} An interesting discovery here  is that
we also need to use the relative Kodaira dimension of a generic
smooth fiber. Here the support of the divisor is the intersection
with a maximal set of disjoint $-1$ spheres, and the coefficients
are negative.
Let $(M', \omega')$ be a minimal model of $(M, \omega)$ and $E_i$
the classes of the symplectic $-1$ spheres in $(M, \omega)$ that are
blown down to obtain $(M , \omega)$. Let $c$ be the number of those $-1$ spheres. Since  $\kappa^s(M,
\omega)=\kappa^s(M', \omega')$, we can compute it using the
expression
$$\kappa^s(K_{\omega'}^2)+\kappa^s(K_{\omega'}\cdot [\omega']).$$

 Now,
let us first compute $\kappa^s(K_{\omega'}^2)$. First, we have
\begin{equation}\label{formula1}K_{\omega'}^2=(K_{\omega}-\sum
E_i)^2=K_{\omega}^2+c.\end{equation}

Notice that we can fiber sum $(M,\omega, F)$ with itself to get a
minimal manifold $(DM, \tau)$. $(DM, \tau)$ also has a genus $g$
hyperelliptic Lefschetz fibration structure with twice of the
singular fibers. It is minimal by the result of \cite{U1}. In addition,
\begin{equation}\label{formula2}
K_{\tau}^2=2(K_{\omega}+[F])^2.
\end{equation}
 Using the hyperelliptic Lefschetz fibration
structure on $(DM, \tau)$, we can also compute $K_{\tau}^2$ by
\eqref{k2} and \eqref{bns},
\begin{equation}\label{formula3}K_{\tau}^2=2(\sum b_i-1)(4g-4)
\end{equation}
Thus, combine \eqref{formula1}, \eqref{formula2}, \eqref{formula3},
we have \begin{equation} \kappa^s(-2+\sum
b_i+\dfrac{c}{4g-4})=\kappa^s(K_{\omega'}^2).\end{equation} Regard
$\kappa^s(-2+\sum b_i+\dfrac{c}{4g-4})$ as the relative Kodaira
dimension of the base. When $c=0$,  this is just what we have
previously.

Now we turn to   $\kappa^s(K_{\omega'}\cdot [\omega'])$. Let $F'$ be
the symplectic surface in $(M', \omega')$ obtained by blowing down
$F$ and smoothing. Let $c'=\sum [F]\cdot [E_i]$, by \eqref {final} applied to $F'\subset (M',
\omega')$, when $K_{\omega'}^2\geq 0$, we have
$$\kappa^s(K_{\omega'}\cdot [\omega'])=\kappa^s(K_{\omega}'\cdot [F'])=\kappa^s(2g-2-c').$$

Thus $\kappa^s(K_{\omega'}\cdot [\omega'])$
 can be viewed as the relative
Kodaira dimension of the fiber relative to the pencil points but
with ``negative mass'', at least when $K_{\omega'}^2\geq 0$. In
particular,  we have in this case
$$\kappa^s(M, \omega)=\kappa^s(-2+\sum
b_i+\dfrac{c}{4g-4})+\kappa^s(2g-2-c').$$ This continues to hold when
$K_{\omega'}^2< 0$ as both sides are equal to $-\infty$.

\end{document}